%% file: dmtcs.tex
\documentclass[final]{dmtcs-episciences}

\usepackage[utf8]{inputenc}
\usepackage{subfigure}
\usepackage[round]{natbib}

\usepackage[british]{babel}

\usepackage{amsmath}
\usepackage{amssymb}
\usepackage{amsthm}

\newtheorem{theorem}{Theorem}
\newtheorem{lemma}[theorem]{Lemma}

\newtheorem{definition}[theorem]{Definition}
\newtheorem{proposition}[theorem]{Proposition}
\newtheorem{remark}[theorem]{Remark}
\newtheorem{example}[theorem]{Example}

\newtheorem{corollary}[theorem]{Corollary}

\usepackage{stmaryrd}

\usepackage{tikz} 
\usetikzlibrary{decorations.pathreplacing, fit, bbox}

\DeclareMathOperator{\spn}{span}

\newcommand{\Z}{\mathbb{Z}}
\newcommand{\N}{\mathbb{N}}
\newcommand{\cL}{\mathcal{L}}
\newcommand{\cA}{\mathcal{A}}
\newcommand{\cB}{\mathcal{B}}
\newcommand{\cG}{\mathcal{G}}
\newcommand{\XX}{\mathbb{X}}
\newcommand{\eps}{\varepsilon}
\newcommand{\orbcl}[1]{\overline{\mathcal{O}(#1)}}
\newcommand{\Card}{\#}
\newcommand{\Occ}[3]{\textrm{Occ}_{#1}^{#2}(#3)}
\newcommand{\Par}[2]{\mathrm{Par}_{#1,#2}}
\newcommand{\IM}[2]{\mathrm{Im}_{#1,#2}}

\newtheorem*{theorem*}{Theorem}

\author{Pierre Béaur\affiliationmark{1}
\and France Gheeraert\affiliationmark{2}
\and Benjamin Hellouin de Menibus\affiliationmark{1}}
\title{String attractors and bi-infinite words}
\affiliation{Laboratoire Interdisciplinaire des Sciences du Numérique, Université Paris-Saclay, CNRS, Orsay, France\\
LAMFA, Université de Picardie Jules Verne, Amiens, France}


\begin{document}

\publicationdata{vol. 28:3}{2026}{21}{10.46298/dmtcs.18407}{2026-06-03; 2026-06-03; 2026-11-02}{2026-11-03}

\maketitle

\begin{abstract}
String attractors are a combinatorial tool coming from the field of data compression. It is a set of positions within a word which captures an occurrence of every factor. While one-sided infinite words admitting a finite string attractor are eventually periodic, the situation is different for two-sided infinite words. In this article, we characterise the bi-infinite words admitting a finite string attractor as the eventually periodic words, the characteristic Sturmian words and their morphic images.
For words that do not admit finite string attractors, we study the structure and properties of their infinite string attractors.
\end{abstract}

\section{Introduction}

A string attractor of a finite word (or sequence of symbols) is a set of positions which captures all the factors of this word, in the sense that every factor has an occurrence crossing one of the positions. Such a set can be related to data compression techniques such as the Lempel-Ziv factorisation or the Burrows-Wheeler transform, see~\cite{prezza_2018}. As such, string attractors (and in particular string attractors of minimal size) provide a common ground between various compression methods and may open the door to new techniques.

Originating in the field of data compression algorithms, string attractors have quickly gained much traction in the combinatorics on words community as a simple enough object having a far from trivial behaviour when faced with operations such as concatenation.
String attractors (for bi-infinite words) also appeared independently (and not explicitly named) as a technical tool in symbolic dynamics, in the study of so-called indistinguishable asymptotic pairs done in~\cite{Barbieri_Labbe_Starosta} and~\cite{barbieri2025indistinguishable}.

Finding a string attractor of minimal size was shown to be an NP-complete problem in general in~\cite{prezza_2018}. Therefore the search quickly turned to particular finite words, such as the prefixes of the Thue-Morse word in~\cite{schaeffer_2021} and~\cite{Kutsukake_et_al} or standard Sturmian words in~\cite{mantaci_2021}. Using combinatorial techniques, string attractors have also been used in~\cite{sa_and_infinite_words} and~\cite{sa_and_infinite_words_extended_version} to characterise Sturmian and quasi-Sturmian (one-sided) infinite words, two famous families of words which can be seen as the simplest interesting infinite words. These works consider string attractors for prefixes of the infinite word, and not for the infinite word itself. The reason for this is simple: the only one-sided infinite words admitting a finite string attractor are eventually periodic
, so this notion has limited interest. Indeed, a fixed position can capture a bounded number of factors of each length, so a finite string attractor implies that the factor complexity is bounded.

However, the situation is different when considering two-sided infinite words, as the number of length-$n$ factors captured by a fixed position is not necessarily bounded. This is a motivation to study bi-infinite words admitting a finite string attractor, which is the main goal of this article.
We obtain the following characterisation of all bi-infinite words admitting a finite string attractor:

\begin{theorem*}[Main result -- Theorem~\ref{T:finite string attractors}]
A bi-infinite word $x$ has a finite string attractor if and only if it falls into one of the following (disjoint) cases:
\begin{itemize}
    \item it is purely periodic;
    \item it is eventually periodic but not purely periodic;
    \item it is, up to finite shift, characteristic quasi-Sturmian (the image of a characteristic Sturmian word by an acyclic substitution).
\end{itemize}

In the last two cases, $x$ admits an interval of length $k+1$ as a string attractor, where $k$ is such that the factor complexity of $x$ is eventually equal to $n + k$.
\end{theorem*}
Let us make a few observations on this result. First, the latter string attractor is as small as possible, as the factor complexity of a word that admits an interval of length $\ell+1$ as a string attractor is upper bounded by $n+\ell$. Second, string attractors made of consecutive positions are also used in the characterisation of (one-sided) Sturmian and quasi-Sturmian words in~\cite{sa_and_infinite_words} and~\cite{sa_and_infinite_words_extended_version}. Third, while Sturmian and quasi-Sturmian words have many equivalent characterisations, combinatorial characterisations of characteristic quasi-Sturmian words do not seem as common, and we discuss some possible links with other characterisations (notably indistinguishable asymptotic pairs from~\cite{Barbieri_Labbe_Starosta}) in Section~\ref{sec:pattern-attractors}. Finally, this result is reminiscent of the decomposition of balanced sequences in four families done in~\cite{Morse_Hedlund}. Indeed, there is a similar decomposition for substitutive images of balanced sequences into four categories: purely periodic, characteristic quasi-Sturmian, non-characteristic quasi-Sturmian, and skew quasi-Sturmian (which are eventually periodic non-purely periodic words). Thus, three of these families have a finite string attractor. This is not a characterisation however as some eventually periodic sequences are not the image of a balanced sequence (see Example~\ref{example:unbalanced}).

\medskip

This article is organised as follows. First, we recall usual notions of combinatorics on words and introduce string attractors in Section~\ref{S:preliminaries}. We improve existing results connecting string attractors and two classical notions, periodicity and substitutions, in Section~\ref{S:classical tools}.

Section~\ref{S:finite string attractors} is devoted to the study of finite string attractors and to the proof of the main result.
For this we study string attractors on three families of words: eventually periodic, Sturmian, and quasi-Sturmian, leading to our main result. We conclude the section by studying finite string attractors in the context of shift spaces and attractors that apply to patterns instead of factors.

Next, we consider in Section~\ref{S:infinite string attractors} words that do not have any finite string attractors, and attempt to define a notion of ``smallest'' infinite string attractor. However, we show in Propositions~\ref{P:Density of infi SA} and~\ref{P:sparse string attractor for shift spaces} that any recurrent word or minimal shift space admits arbitrarily sparse string attractors, implying that such a notion does not exist. We therefore turn our attention to the study of bi-infinite words and shift spaces having particular infinite string attractors (every arithmetic progression) and relate this to other dynamical properties.

\section{Preliminaries}\label{S:preliminaries}

\subsection{Words, shift spaces and substitutions}

An \emph{alphabet} $\cA$ is a finite set of symbols, called \emph{letters}. The set of \emph{finite words}, i.e., of finite sequences, over $\cA$ is denoted $\cA^*$ and it is naturally endowed with the concatenation. The \emph{length} of a finite word $w$, denoted $|w|$, is the integer $\ell$ such that $w \in \cA^\ell$; the empty word $\eps$ is the only word of length $0$. Finite words are indexed from $0$ to $|w|-1$. The \emph{one-sided infinite words} over $\cA$ are the words of $\cA^\N$, and the \emph{bi-infinite words} over $\cA$ are the words of $\cA^\Z$. For $(x_i)_{i \in \Z} \in \cA^\Z$, we sometimes write $x = \cdots x_{-3} x_{-2} x_{-1}.x_0 x_1 x_2 \cdots$. We extend concatenation to infinite words with the convention that $xw = x$ for $x \in \cA^\N \cup \cA^\Z$ and $w \in \cA^\ast$.

A finite word $u$ is a \emph{factor} of a finite or infinite word $w$ if there exist words $p,s$ such that $w = pus$. If $p = \eps$ (resp., $s = \eps$), we moreover say that $u$ is a \emph{prefix} (resp., \emph{suffix}) of $w$. If $u\neq w$, we also say that $u$ is \emph{proper}. The set of length-$n$ factors of $w$ is denoted $\cL_n(w)$ and the set of all factors of $w$ is denoted $\cL(w)$. The \emph{factor complexity} of a word $x$ is the function $p_x$ such that $p_x(n) = \Card{\cL_n(x)}$ for all $n$. 
A factor $u$ of $w$ is \emph{left-special} (resp., \emph{right-special}) if there exist (at least) two different letters $a$ and $b$ such that $au, bu \in \cL(x)$ (resp., $ua, ub \in \cL(x)$). It is \emph{bispecial} if it is both left-special and right-special.

We use the notation $\llbracket j,k\rrbracket$ for $\{j, j+1,\dots, k\}$ and $x_{\llbracket j,k \rrbracket}$ for $x_jx_{j+1}\cdots x_k$.

\begin{definition}[Occurrence, crossing] 
For a word $x$, the factor $x_{\llbracket j,k\rrbracket}$ has an \emph{occurrence} in $x$ starting in $j$, ending in $k$ and \emph{crossing} every position in $\llbracket j,k\rrbracket$.
\end{definition}

A finite word $x$ is \emph{periodic of period $p \geq 1$} (or \emph{$p$-periodic} for short) if $x_n = x_{n+p}$ for all $n \in \llbracket 0, |x|-p-1\rrbracket$. We then also say that it is a \emph{fractional power} of $x_{\llbracket 0,p-1\rrbracket}$.
A bi-infinite word $x$ is:
\begin{itemize}
    \item \emph{(purely) periodic} if there exists a period $p \geq 1$ such that, for all $n \in \Z$, $x_n = x_{n+p}$;
    \item \emph{eventually periodic} if there are periods $p,q \geq 1$ and indices $N_p, N_q\in \Z$ such that $x_n = x_{n+p}$ for all $n\geq N_p$ and $x_n = x_{n-q}$ for all $n\leq N_q$; 
    \item \emph{aperiodic} if for all $p \geq 1$ and $N\in \Z$ there exist $n\leq N$ and $n'\geq N$ such that $x_n \neq x_{n+p}$ and $x_{n'} \neq x_{n'+p}$.
\end{itemize} 
A bi-infinite word $x$ is \emph{recurrent} if every factor has infinitely many occurrences, and it is \emph{uniformly recurrent} if every factor occurs with bounded gaps, i.e., for each $u \in \cL(x)$, there exists $n$ such that $u$ is a factor of every $v \in \cL_n(x)$.

\begin{remark}\label{R:uniformly recurrent and aperiodic imply periodic}
If $x \in \cA^\Z$ is uniformly recurrent and not aperiodic, then it is purely periodic.
\end{remark}

\medskip

A \emph{shift space} is a subset of $\cA^\Z$, closed for the product topology (when $\cA$ is endowed with the discrete topology) and stable under the \emph{shift map} $S\colon (x_i)_{i \in \Z} \mapsto (x_{i+1})_{i \in \Z}$. 
The \emph{orbit} of a word $x \in \cA^\Z$ is the set $\{S^k(x) \mid k \in \Z\}$. Its topological closure is denoted $\orbcl{x}$ and is a shift space.

A shift space is \emph{aperiodic} if it does not contain any purely periodic words. It is \emph{minimal} if the only shift spaces it contains are itself and $\emptyset$. Equivalently, a shift space is minimal if and only if it is the orbit closure of each of its words. All words of a minimal shift space are uniformly recurrent and share the same language.

\medskip

A \emph{substitution} is a monoid morphism $\varphi$ from $\cA^*$ to $\cB^*$ such that $\varphi(u) \neq \eps$ for any $u \in \cA^* \setminus \{\eps\}$ (this is sometimes called a \emph{non-erasing morphism}). It is entirely determined by the images of the elements of $\cA$ since $\varphi(w)=\varphi(w_0)\cdots\varphi(w_{|w|-1})$. We naturally extend substitutions to infinite words by concatenation. 
We also define the image of a shift space as follows: if $X \subseteq \cA^\Z$ is a shift space, then $\varphi(X) = \{S^k(\varphi(x)) \mid x \in X, 0 \leq k < |\varphi(x_0)|\}$.

\subsection{String attractors}

\begin{definition}[String attractor]
    Let $x$ be a word and let $\Gamma \subseteq \Z$ be a set of positions within $x$. A factor $w$ of $x$ is \emph{captured by $\Gamma$} if it has an occurrence crossing $\Gamma$, i.e., there exists $i$ such that $x_{\llbracket i, i+ |w|-1 \rrbracket} = w$ and $\llbracket i, i + |w| -1 \rrbracket \cap \Gamma \neq \emptyset$.
    If every non-empty factor of $x$ is captured by $\Gamma$, we say that $\Gamma$ is a \emph{string attractor} of $x$.
\end{definition}

\begin{example}\label{example:unbalanced}
Consider the word $x = \cdots 000\underline{0.1}111 \cdots$, i.e., $x_i = 0$ if $i < 0$ and $x_i = 1$ otherwise.
Its language is $\mathcal{L}(x) = \{0^i1^j \mid i,j \geq 0 \}$.
Consider $\Gamma = \llbracket -1,0 \rrbracket$ (underlined) and let $0^i1^j$ be a nonempty factor of $x$. We have $x_{\llbracket-i,j-1\rrbracket} = 0^i1^j$, and since $0^i1^j$ is nonempty, $\llbracket -i, j-1 \rrbracket \cap \llbracket -1, 0 \rrbracket \neq \emptyset$.
Thus $0^i1^j$ is captured by $\Gamma$. As this is true for every nonempty factor of $x$, $\Gamma$ is a string attractor of $x$.
\end{example}

In this work, we encounter the following type of string attractors:

\begin{definition}[Perfect, eventually perfect attractors]
A string attractor $\Gamma$ of $x$ is \emph{perfect} if every factor of $x$ has exactly one occurrence crossing $\Gamma$. It is \emph{eventually perfect} if this holds for all long enough factors.
\end{definition}

We also consider string attractors of shift spaces:

\begin{definition}
    Let $\mathbb{X}$ be a shift space.
    A set $\Gamma \subseteq \Z$ is a \emph{string attractor of} $\mathbb{X}$ if $\Gamma$ is a string attractor of every $x \in \mathbb{X}$.
\end{definition}

\section{String attractors toolbox}\label{S:classical tools}

\subsection{Periodicity}

Periodicity is known to help with the study of string attractors. In particular, a version of the following simple result already appears in~\cite{mantaci_2021}.

\begin{lemma}\label{L:one period}
    Let $w$ be a finite $p$-periodic word. If $\Gamma$ is a string attractor of $w_{\llbracket 0, p-1\rrbracket}$, then $\Gamma \cup \{p-1\}$ is a string attractor of $w$.
    In particular, $\llbracket 0, p-1\rrbracket$ is a string attractor of $w$.
    By symmetry, $\llbracket |w|-p, |w|-1\rrbracket$ is also a string attractor of $w$.
\end{lemma}

When the word admits two (unrelated) periods, we can combine the string attractors of Lemma~\ref{L:one period} to obtain a smaller string attractor.

\begin{lemma}\label{L:string attractor of finite periodic word - new}
    If $w$ is a finite word that is both $p$-periodic and $q$-periodic but not $\gcd(p,q)$-periodic, then $\Gamma = \llbracket |w|-q, p-1\rrbracket$ is a string attractor of $w$.
\end{lemma}
\begin{proof}
    Observe that, by non-periodicity for $\gcd(p,q)$ and by Fine and Wilf's theorem, $p$ and $q$ are not multiple of one another and we have $|w| < p + q - \gcd(p,q)$. In particular, $p > |w| - q$ so $\llbracket |w|-q, p-1\rrbracket$ is a valid non-empty interval.

    We proceed by induction on $|w|-|f|$ to prove that, for all such $w$ and for every non-empty factor $f$ of $w$, $f$ is captured by $\Gamma$. For the base case, if $|f| = |w|$ then $f = w$ which has an occurrence crossing $\Gamma$. Let us prove the inductive case for some $w$ and some factor $f$ of $w$. We assume that $p < q$, the other case is symmetric.
    
    Write $w = uw'v$ where $u$ is the length-$(|w|-q)$ prefix, $v$ is the length-$(|w|-p)$ suffix. Note that this occurrence of $w'$ corresponds to the positions of $\Gamma$ in $w$. By $p$-periodicity and $q$-periodicity (and since $p < q = |w| - |u|$), $u$ has occurrences starting at positions $p$ and $q$. So $u$ is both a prefix and a suffix of $v$ and of $w$.
    
    If $f$ is not a suffix of $w$, there is a letter $a$ such that $fa$ is a factor of $w$ and crosses $\Gamma$ by induction hypothesis. Therefore, $f$ is immediately captured by $\Gamma$, unless $fa$ crosses $\Gamma$ on the last letter only. This last case is impossible as it implies that $f$ is a suffix of $u$ and therefore of $w$. If $f$ is a suffix of $w$ and $|f| > |v|$, then $f$ obviously crosses $\Gamma$. Thus the inductive case is proved except if in the case where $f$ is a suffix of $v$, which we consider now. 
    
    By $p$-periodicity, $v$ has an occurrence starting at position $0$. If $|v| \leq p$, this occurrence of $v$ ends in $\Gamma$ since $|u| < |v|$. Therefore the suffix $f$ has an occurrence crossing $\Gamma$. 
    Now, if $|v| > p$, write $v = uw'v'$ where $v'$ is a prefix of $v$. Since $u$ is both a prefix and a suffix of $v$, it follows that $v$ is periodic of period $|v| - |u| = q - p$, and also $p$-periodic.
    Note that $\gcd(p,q-p) = \gcd(p,q)$. 
    For the sake of contradiction, assume that $v$ is $\gcd(p,q)$-periodic. Since $w = uw'v$ is $p$-periodic, it is a fractional power of $uw'$, which is $\gcd(p,q)$-periodic (as a prefix of $v$). Since $\gcd(p,q)$ divides $|uw'| = p$, it follows that $w$ is $\gcd(p,q)$-periodic, a contradiction.
    
    As $|v| - |f| < |w| - |f|$, applying the induction hypothesis on $v$ (with periods $p$ and $q-p$) and its factor $f$, we find an occurrence of $f$ in $v$ crossing $\llbracket |v| - q + p, p-1\rrbracket = \llbracket |w| - q, p-1\rrbracket = \Gamma$. We conclude that, in this final case, $f$ is also captured by $\Gamma$ in $w$ since $v$ is a prefix of $w$.
\end{proof}

\subsection{Substitutions}

As already observed in \cite{sa_and_infinite_words}, a string attractor of $x$ induces a string attractor of the image $\varphi(x)$ of $x$ under a substitution $\varphi$. To formalise this, let us introduce some notation linking positions in $x$ and positions in $\varphi(x)$.

\begin{definition}\label{D:image of sa}
    Let $x \in \cA^*$ be a word  and let $\varphi\colon \cA^* \to \cB^*$ be a substitution.
    We define the \emph{parent of a position} $j$ in $\varphi(x)$ as follows:
    \[
        \Par{\varphi}{x}(j) = i \text{ such that } |\varphi(x_{\llbracket 0, i-1\rrbracket})| \leq j < |\varphi(x_{\llbracket 0, i\rrbracket})|.
    \]
    Conversely, the image of a position $i$ in $x$ is
    \[
        \IM{\varphi}{x}(i) =
        \left\llbracket |\varphi(x_{\llbracket 0, i-1\rrbracket})|, |\varphi(x_{\llbracket 0, i\rrbracket})| - 1 \right\rrbracket.
    \]
\end{definition}

Conceptually, $\Par{\varphi}{x}(j)$ is the position $i$ such that the image of $x_i$ crosses the position $j$ in $\varphi(x)$, and $\IM{\varphi}{x}(i) = \Par{\varphi}{x}^{-1}(i)$ is the set of positions crossed by $\varphi(x_i)$ in $\varphi(x)$. We naturally extend these notions to the case where $x$ is (one-sided) infinite or bi-infinite.

Notice that, as $\varphi$ is non-erasing, if $\Gamma$ is an interval, then $\Par{\varphi}{x}(\Gamma)$ and $\IM{\varphi}{x}(\Gamma)$ are also intervals. Furthermore, the image of a string attractor is a string attractor, as can be deduced from the proof of~\cite[Proposition 4]{sa_and_infinite_words}:

\begin{proposition}\label{P:under_morphism}
    Let $x$ be a word and $\varphi\colon \cA^* \to \cB^*$ be a substitution. If $\Gamma$ is a string attractor of $x$, then $\IM{\varphi}{x}(\Gamma)$ is a string attractor of $\varphi(x)$. 
\end{proposition}

Conversely, $\Gamma$ being a string attractor of $\varphi(x)$ does not imply in general that $\Par{\varphi}{x}(\Gamma)$ (or a small variation of it) is a string attractor of $x$. For example, it is false when $\varphi(x)$ is periodic and $x$ is not. However, it can hold under some hypothesis on $\varphi$. In the context of this work, we focus on a specific family of substitutions.

\begin{definition}[Return morphism]
    Let $w \in \mathcal{A}^*$ be non-empty. A substitution $\varphi\colon \mathcal{A}^\ast \rightarrow \mathcal{B}^\ast$ is a \emph{return morphism for $w$} if $\varphi$ is injective on $\mathcal{A}$ and, for all $a \in \mathcal{A}$, $\varphi(a)w$ contains exactly two distinct occurrences of $w$: one as a prefix and one as a suffix. We say that $\varphi$ is a \emph{return morphism} if it is a return morphism for some $w$.
\end{definition}

\begin{example}
    The substitution $\varphi\colon 0 \mapsto 01, 1 \mapsto 010, 2 \mapsto 01020$ is a return morphism for $010$ as $01010$, $010010$ and $01020010$ all start and end with $010$ but contain no internal occurrence of $010$. Note that it is also a return morphism for $01$.
\end{example}

Return morphisms are strongly related to the notions of return words and of derivation (introduced in~\cite{Durand_substitutive}) as the elements of $\varphi(\cA)$ are the return words for $w$ in $\varphi(x)$. Consequently, they possess many properties, especially in regard to recognisability (see~\cite[Chapter 4]{TheseFrance} for example). We list below the properties used in this work.

\begin{lemma}\label{L:return morphisms}
    Let $\varphi\colon \cA^* \to \cB^*$ be a return morphism for $w$ and let $x$ be a word on the alphabet $\cA$.
    \begin{enumerate}
        \item For any $u \in \cA^*$, $w$ is a prefix of $\varphi(u)w$. In other words, $\varphi(u)w$ is $|\varphi(u)|$-periodic if $u \ne \varepsilon$.
        \item For any $u \in \cL(x)$, occurrences of $u$ in $x$ correspond exactly to occurrences of $\varphi(u)w$ in $\varphi(x)w$. More precisely, $\varphi(u)w$ has an occurrence in $\varphi(x)w$ such that the $\varphi(u)$ part crosses position $j$ if and only if $u$ has an occurrence in $x$ crossing $\Par{\varphi}{x}(j)$.
        \item The substitution $\varphi$ is injective (on finite and infinite words).
        \item Assume that $\#\cA = 2$. If $p$ is the longest common prefix to all $\varphi(a)w$, $a \in \cA$, then for any right-special $u \in \cL(x)$, $\varphi(u)p$ is right-special in $\varphi(x)w$. Similarly, if $s$ is the longest common suffix to all $\varphi(a)$, $a \in \cA$, then for any left-special $u \in \cL(x)$, $s\varphi(u)w$ is left-special in $\varphi(x)w$.
    \end{enumerate}
\end{lemma}

Moreover, any return morphism is a return morphism for a right-special word: it suffices to take the longest word for which it is a return morphism (see~\cite[Proposition 6]{GheeraertLeroy}).

For return morphisms, we then have a stronger version of Proposition~\ref{P:under_morphism}. To prove it, we use the following lemma which will also play a role for eventually periodic words in Section~\ref{ss:eventually periodic}.

\begin{lemma}\label{lem:smallfactors}    
    Let $x$ be a word. Assume that $\Gamma = \llbracket n, n+k \rrbracket$ captures all factors of length at least $\ell$. If $x_{\llbracket n+k+1, n+k+\ell - 1\rrbracket}$ is left-special and $x_{\llbracket n - \ell + 1, n - 1\rrbracket}$ is right-special, then $\Gamma$ is a string attractor of $x$.
\end{lemma}

\begin{proof}
    We proceed by induction on $\ell$. The result is trivially true for $\ell = 1$. Assume the result holds for some $\ell \geq 1$. Assume also that $\Gamma$ captures all factors of length at least $\ell+1$ and that $x_{\llbracket n+k+1, n+k+\ell\rrbracket}$ is left-special and $x_{\llbracket n - \ell, n - 1\rrbracket}$ is right-special.

    Every length-$\ell$ factors appear as prefix or suffix of some length-$(\ell+1)$ factor. Take a length-$\ell$ factor $f$ and assume that it is a suffix of a length-$(\ell+1)$ factor. By hypothesis, $\Gamma$ captures $f$ except possibly if $f = x_{\llbracket n+k+1, n+k+\ell \rrbracket}$. In that case, $f$ is left-special and is also a suffix of some length-$(\ell+1)$ factor $f' \neq x_{\llbracket n+k, n+k+\ell\rrbracket}$. By hypothesis, $f'$ is captured by $\Gamma$ at some position other than $n+k$, so $f$ is captured by $\Gamma$. The argument when $f$ is a prefix of a length-$(\ell+1)$ factor is symmetric, using the fact that $x_{\llbracket n - \ell, n - 1\rrbracket}$ is right-special. Thus, $\Gamma$ captures all length-$\ell$ factors. 
    
    Moreover, since $x_{\llbracket n+k+1, n+k+\ell-1 \rrbracket}$ is a prefix of $x_{\llbracket n+k+1, n+k+\ell\rrbracket}$, it is also left-special, and similarly, $x_{\llbracket n - \ell + 1, n - 1\rrbracket}$ is right-special.
    By induction hypothesis, $\Gamma$ is a string attractor. 
\end{proof}

\begin{remark}
\label{rem:smallfactors}
    Observe that, if $x$ is bi-infinite, then it is not necessary to have both $x_{\llbracket n+k+1, n+k+\ell - 1\rrbracket}$ left-special and $x_{\llbracket n - \ell + 1, n - 1\rrbracket}$ right-special : one of the two condition suffices. Observe also that, if $x$ is finite and $n+k+\ell-1 \geq |x|$, then we can replace the condition ``$x_{\llbracket n+k+1, n+k+\ell - 1\rrbracket}$ is left-special''  by ``$x_{\llbracket n+k+1, |x|-1\rrbracket}$ is left-special''. Similarly, if $n-\ell+1 < 0$, then we can replace ``$x_{\llbracket n - \ell + 1, n - 1\rrbracket}$ is right-special'' by ``$x_{\llbracket 0, n - 1\rrbracket}$ is right-special''.
\end{remark}

\begin{lemma}\label{L:image of sa return morphism}
    Let $x$ be a word on $\cA$ using at least two different letters, and $\varphi\colon \cA^* \to \cB^*$ be a return morphism for a bispecial word $w$. If $\Gamma = \llbracket n, n+k \rrbracket$ is a string attractor of $x$, then $\IM{\varphi}{x}(\Gamma)$ without its first $|w|$ positions is a string attractor of $\varphi(x)w$.
\end{lemma}

\begin{proof}
    Let us first prove that $\IM{\varphi}{x}(\Gamma)$ contains at least $|w|+1$ positions.
    Let $a, b$ be two distinct letters appearing in $x$.
    As $\Gamma$ captures all letters of $x$, we deduce that $\# \IM{\varphi}{x}(\Gamma) \geq |\varphi(a)| + |\varphi(b)| = |\varphi(ab)|$.
    Since $\varphi$ is a return morphism for $w$, $\varphi(ab) \ne \varphi(ba)$ and $w$ is a prefix of both $\varphi(ab)w$ and $\varphi(ba)w$. In particular, $w$ is actually a proper prefix of $\varphi(ab)$ and $\varphi(ba)$.
    It follows that $|\varphi(ab)| \geq |w| + 1$. Let $\Gamma'$ denote the set $\IM{\varphi}{x}(\Gamma)$ without its first $|w|$ positions. Thus, $\Gamma' = \llbracket m, m+\ell \rrbracket$ for some $m$ and $\ell$, and $\varphi(x)_{\llbracket m-|w|, m-1\rrbracket} = w = \varphi(x)_{\llbracket m+\ell+1, m+\ell+|w|\rrbracket}$.

    \begin{figure}[h!t]
        \centering
        \input{image_return_1}
        \caption{Contradictory case where the occurrence of $u$ ends in the $|w|$ first positions of $\IM{\varphi}{x}(\Gamma)$. The grey portion represents the positions in $\Gamma'$.}
        \label{F:image return morphism - case 1}
    \end{figure}
    
    By Lemma~\ref{lem:smallfactors}, it is sufficient to show that $\Gamma'$ captures all factors of length at least $|w|+1$ since $w$ is bispecial. Take $u$ one such factor and let $t \in \cL(x)$ be a shortest word such that $u$ is a factor of $\varphi(t)w$. As $u$ is long enough, $t$ is not empty. Thus, $t$ has an occurrence in $x$ crossing the string attractor $\Gamma$. This corresponds to an occurrence of $\varphi(t)w$ in $\varphi(x)w$, where the $\varphi(t)$ part crosses $\IM{\varphi}{x}(\Gamma)$. By minimality of $t$, $u$ has an occurrence in $\varphi(x)w$ crossing $\IM{\varphi}{x}(\Gamma)$. Moreover, this occurrence cannot end in the first $|w|$ positions of $\IM{\varphi}{x}(\Gamma)$. Indeed, as these positions correspond to an occurrence of $w$, this would mean that $u$ is a factor of $\varphi(t')w$ for $t'$ a proper prefix of $t$, which contradicts the minimality of $t$. This hypothetical situation is represented in Figure~\ref{F:image return morphism - case 1}. Hence, $u$ has an occurrence in $\varphi(x)w$ crossing $\Gamma'$.
\end{proof}

For return morphisms, the converse result now holds. To simplify the statement, for $\Gamma \subseteq \Z$, $\Par{\varphi}{x}(\Gamma)$ should be understood as $\Par{\varphi}{x}(\Gamma \cap I)$ where $I$ the support of $\varphi(x)$.

\begin{lemma}\label{L:inverse image of sa return morphism}
    Let $x$ be a word, $\varphi\colon \cA^* \to \cB^*$ be a return morphism for $w$, and $\Gamma = \llbracket n, n+k \rrbracket$ be a set of positions in $\varphi(x)$. If $\Gamma$ is a string attractor of $\varphi(x)w$, then $\Par{\varphi}{x}(\llbracket n-|w|,n+k \rrbracket)$ is a string attractor of $x$.
\end{lemma}

\begin{proof}
    Let $u$ be a non-empty factor of $x$. Since $\Gamma$ is a string attractor of $\varphi(x)w$, there is an occurrence of $\varphi(u)w$ in $\varphi(x)w$ crossing $\Gamma$. We consider two cases. If the $\varphi(u)$ part crosses a position $j \in \llbracket n, n+k \rrbracket$, then there is an occurrence of $u$ in $x$ crossing $\Par{\varphi}{x}(j) \in \Par{\varphi}{x}(\llbracket n-|w|,n+k \rrbracket)$. Otherwise, the $w$ part of $\varphi(u)w$ crosses position $n$. Thus, the $\varphi(u)$ part crosses a position $j \in \llbracket n-|w|, n-1 \rrbracket$, and there is an occurrence of $u$ in $x$ crossing position $\Par{\varphi}{x}(j)\in \Par{\varphi}{x}(\llbracket n-|w|,n+k \rrbracket)$.
\end{proof}

\section{Finite string attractors}\label{S:finite string attractors}

Any finite word trivially admits a finite string attractor. The situation is different for one-sided infinite words, as seen in the result below which is an alternative formulation of~\cite[Proposition~6]{sa_and_infinite_words}.

\begin{proposition}
A one-sided infinite word $x$ admits a finite string attractor if and only if $x$ is eventually periodic, i.e., there exist $p\geq 1$ and $N \in \N$ such that $x_{n+p} = x_n$ for all $n \geq N$.
\end{proposition}

For bi-infinite words however, the answer is not as black and white; answering it is the goal of this section. To do so, we look at the minimal span of a string attractor. 

\begin{definition}[Span]
    The \emph{span} of a set $\Gamma \subseteq \Z$ is defined as $\spn(\Gamma) = \sup \Gamma - \inf \Gamma$. The \emph{(string attractor) span} of a word $x$ is defined as
    \[
        \spn(x) = \inf \{\spn(\Gamma) \mid \Gamma \text{ string attractor of } x\}.
    \]
    In particular, $\spn(x)$ is finite if and only if $x$ admits a finite string attractor.
\end{definition}

The span was first introduced for finite words in~\cite{sa_and_infinite_words} where the authors study some of its combinatorial properties. Many of their results can be adapted to infinite words. For example, the following proposition shows the link between span and factor complexity, and is a direct adaptation of~\cite[Lemma 1]{sa_and_infinite_words}. We give the proof here for the sake of completeness.

\begin{proposition}\label{P:finite SA bound comp}
    For any word $x$, we have $p_x(n) \leq n+\spn(x)$ for all $n \geq 1$. 
\end{proposition}
\begin{proof}
    If $\spn(x)$ is infinite, we directly have $p_x(n) \leq n + \spn(x)$ for all $n$.
    Assume that $x$ admits a finite string attractor $\Gamma \subseteq \llbracket \gamma, \gamma + \spn(\Gamma)\rrbracket$ for some $\gamma \in \Z$. Let $n \geq 1$. Since every $w\in \cL_n(x)$ is captured by $\Gamma$, it has an occurrence starting in $\llbracket \gamma-n+1, \gamma+\spn(\Gamma)\rrbracket$. This implies that $p_x(n) \leq \gamma+\spn(\Gamma) - (\gamma-n+1)+1 = n+\spn(\Gamma)$. As this is true for any string attractor $\Gamma$, $p_x(n) \leq n + \spn(x)$.
\end{proof}

\begin{remark}
Observe that if $p_x(n) = n + k$ for all $n \geq n_0$ for some $n_0\geq 1$ and if $x$ has a string attractor $\Gamma = \llbracket \gamma, \gamma + s\rrbracket$, then this string attractor is eventually perfect if and only if $s = k$. More precisely, the length-$n$ factors are then each captured exactly once by $\Gamma$ for all $n \geq n_0$.
\end{remark}

By Proposition~\ref{P:finite SA bound comp}, if a bi-infinite word $x$ has a finite string attractor (or equivalently, if its span is finite), then its factor complexity is bounded by the linear function $n+\spn(x)$.
For such words, we have the following classical dichotomy.

\begin{proposition}[\cite{CovenHedlund,Coven1975}]\label{P:Coven bi-infinite words}
    Let $x$ be a bi-infinite word. There exists $k \in \N$ such that $p_x(n) \leq n+k$ for all $n$ if and only if $x$ falls into one of the following (mutually exclusive) cases:
    \begin{itemize}
        \item $x$ is eventually periodic;
        \item $x$ is aperiodic, uniformly recurrent, and there exist $n_0 \geq 0$ and $k' \geq 1$ such that $p_x(n) = n + k'$ for all $n \geq n_0$.
    \end{itemize}
\end{proposition}

Words falling in the second case are called quasi-Sturmian in~\cite{Cassaigne_grouped_factors}.

\begin{definition}[Quasi-Sturmian word]
    A bi-infinite word $x$ is \emph{quasi-Sturmian} if it is aperiodic and there exist $n_0 \geq 0$ and $k \geq 1$ such that for every $n \geq n_0$, $p_x(n) = n + k$.
\end{definition}

In particular, quasi-Sturmian words are uniformly recurrent. When $k = 1$, we can show that $n_0 = 0$ and we get the well-known Sturmian words.

\begin{definition}[Sturmian word]
    A bi-infinite word $x$ is \emph{Sturmian} if it is aperiodic and its factor complexity is given by $p_x(n) = n+1$.
\end{definition}

All quasi-Sturmian words are substitutive images of Sturmian words (see Proposition~\ref{P:quasi-Sturmian morphism} for a formal statement).
Therefore, to determine which words admit a finite string attractor among the words of complexity at most $n+k$, we successively consider three families of words: eventually periodic, Sturmian, and quasi-Sturmian. 

\subsection{The eventually periodic case}\label{ss:eventually periodic}

It is easy to see that eventually periodic words admit finite string attractors. Indeed, extending Lemma~\ref{L:one period}, if $x = S^k(\cdots uuu. w vvv\cdots)$, then $\llbracket -|u| - k, |wv| - k-1\rrbracket$ is a string attractor. Let us consider their span.

\begin{remark}\label{R:periodic}
If $x \in \cA^\Z$ is purely $p$-periodic, then $x$ admits the string attractor $\llbracket i, i+p-1 \rrbracket$ for any $i \in \Z$. In particular, $\spn(x) \leq p-1$. However, the value of $\spn(x)$ depends on the exact word and not only on the period: for example, consider the 3-periodic words $x$ and $y$, where $x_{\llbracket 0, 2\rrbracket} = 012$ and $y_{\llbracket 0,2\rrbracket} = 010$. We have $\spn(x) = 2$ and $\spn(y) = 1$.
\end{remark}

For non-purely periodic words however, span and factor complexity are closely related. The next result is illustrated by Example~\ref{example:unbalanced} which is a word with span $1$ and complexity $n+1$.

\begin{proposition}\label{P:sa of periodic words}
    Let $x$ be a bi-infinite word. If $x$ is eventually periodic but not purely periodic, then $x$ admits a finite string attractor and its factor complexity is eventually equal to $n+\spn(x)$.
\end{proposition}

\begin{proof}
Let $x_{\llbracket i, +\infty\llbracket}$ denote the maximal periodic suffix of $x$ and $p$ its minimal period; in particular, all factors $x_{\llbracket i, i+k\rrbracket}$ for $k>0$ are left-special. 
Similarly, let $x_{\rrbracket-\infty,j\rrbracket}$ be the maximal periodic prefix of $x$ and $q$ its minimal period. 
Finally, let $k \geq 1$ be such that $p_x(n) = n+k$ for any large enough $n$. By Proposition~\ref{P:finite SA bound comp}, to show that $\spn(x) = k$, it suffices to find a string attractor $\Gamma$ such that $\spn(\Gamma) = k$. By~\cite[Theorem B]{these_Heinis}, we have $k = i-j+p+q-2$ if the words $x_{\llbracket j-q+1,j\rrbracket}$ and $x_{\llbracket i, i+p-1\rrbracket}$ are not (cyclic) conjugates, and $k = i-j+p-2$ if they are (that is, $p = q$ and there exist $u, v$ such that $x_{\llbracket j-q+1,j\rrbracket} = uv$ and $x_{\llbracket i, i+p-1\rrbracket} = vu$).

In the former case, since $k \geq 1$, we have $j - q + 1 < i + p - 1$. Let us show that $\Gamma = \llbracket j - q + 1, i + p - 1\rrbracket$ is a string attractor of $x$.
Take $n \geq \max(p,q)$. By $p$-periodicity, every length-$n$ factor having an occurrence starting in $\llbracket i+p, +\infty \llbracket$ has an occurrence crossing position $i + p - 1$, and is therefore captured by $\Gamma$. Similarly, every length-$n$ factor having an occurrence ending in $\rrbracket{-\infty}, j-q\rrbracket$ has an occurrence crossing position $j - q + 1$ and is captured by $\Gamma$. All other length-$n$ factors are trivially captured by $\Gamma$. This shows that $\Gamma$ captures all long enough factors. Since $x_{\llbracket i+p, i+p+n\rrbracket} = x_{\llbracket i, i+n\rrbracket}$ is left-special and $x_{\llbracket j-q-n, j-q\rrbracket} = x_{\llbracket j-n, j\rrbracket}$ is right-special, using Lemma~\ref{lem:smallfactors}, $\Gamma$ is then a string attractor of $x$.

In the case where $x_{\llbracket j-q+1,j\rrbracket}$ and $x_{\llbracket i, i+p-1\rrbracket}$ are conjugates, we have $j + 1 < i + p - 1$. We use the same argument as above to show that $\Gamma' = \llbracket j+1,i+p-1\rrbracket$ capture all factors of length $n>p$, with the modification that factors having an occurrence ending in $\rrbracket {-\infty}, j\rrbracket$ also have an occurrence starting in $\llbracket i,i+p-1 \rrbracket$ by conjugacy, so these factors are captured by $i+p-1$.
\end{proof}

\subsection{The Sturmian case}\label{S:Sturmian}

To study string attractors of Sturmian words, we first introduce some notation.
Due to the link between factor complexity and left- (resp., right-) special factors, words of complexity $n+1$ have exactly one left- (resp., right-) special factor of each length.

\begin{definition}
    Let $x$ be a bi-infinite word of factor complexity $p_x(n) = n + 1$.
    Let $\ell_n(x)$ denote the only left-special factor of $x$ of length $n$, and $r_n(x)$ the only right-special factor of $x$ of length $n$.
\end{definition}

Observe that, for all $n \leq m$, $r_n(x)$ is a suffix of $r_m(x)$ and $\ell_n(x)$ is a prefix of $\ell_m(x)$. We then have the following characterisation of all non-purely periodic words having the string attractor $\llbracket 0,1\rrbracket$.

\begin{proposition}\label{P:unicity of words of span 1}
     A bi-infinite word $x \in \{0,1\}^\Z$ with factor complexity $n+1$ has the string attractor $\llbracket 0,1 \rrbracket$ if and only if $x_{\llbracket -n, n+1\rrbracket} = r_n(x)01\ell_n(x)$ for all $n$ or $x_{\llbracket -n, n+1\rrbracket} = r_n(x)10\ell_n(x)$ for all $n$.
\end{proposition}

\begin{proof}
    Observe that, since $p_x(n) = n+1$, $\llbracket 0,1 \rrbracket$ is a string attractor of $x$ if and only if it is a perfect string attractor, or equivalently, if and only if every length-$(n+1)$ factor of $x$ occurs exactly once in $x_{\llbracket -n,n+1\rrbracket}$ for all $n$.
    We show by induction on $N$ that $x_{\llbracket-N,N+1\rrbracket} \in r_N(x)\{01,10\}\ell_N(x)$ if and only if each length-$(n+1)$ factor occurs exactly once in $x_{\llbracket -n,n+1\rrbracket}$ for all $n \leq N$. For $N = 0$, it is clear that every letter occurs in $x_{\llbracket0,1\rrbracket}$ if and only if $x_{\llbracket0,1\rrbracket}$ belongs to $\{01, 10\}$. Now assume that the equivalence is true for $N-1$.
    
    $(\Rightarrow)$ Assume that, for all $n \leq N$, every length-$(n+1)$ factor occurs exactly once in $x_{\llbracket -n, n+1\rrbracket}$. In particular, $x_{\llbracket0,1\rrbracket} \in \{01, 10\}$ (with $n = 0$), and $\ell_N(x)$ occurs exactly once $x_{\llbracket-N+1,N\rrbracket}$ (with $n = N-1$). 
    By definition of $\ell_N(x)$, $0\ell_N(x)$ and $1\ell_N(x)$ are length-$(N+1)$ factors of $x$ so, by hypothesis, they both occur in $x_{\llbracket-N,N+1\rrbracket}$. 
    Thus, $\ell_N(x)$ occurs twice in $x_{\llbracket-N+1,N+1\rrbracket}$ and exactly once in $x_{\llbracket-N+1,N\rrbracket}$. Therefore $x_{\llbracket2,N+1\rrbracket} = \ell_N(x)$.
    We similarly show that $x_{\llbracket-N,-1\rrbracket} = r_N(x)$, which proves that $x_{\llbracket-N,N+1\rrbracket} \in r_N(x)\{01,10\}\ell_N(x)$.
    
    $(\Leftarrow)$ Assume that $x_{\llbracket-N,N+1\rrbracket} = r_N(x) 01 \ell_N(x)$, the other case is symmetric. 
    
    By induction hypothesis, as $x_{\llbracket -N+1,N\rrbracket} = r_{N-1}(x)01\ell_{N-1}(x)$, it suffices to show that each length-$(N+1)$ factor occurs once in $x_{\llbracket-N,N+1\rrbracket}$.
    
    Take $w \in \cL_{N+1}(x)$ and let $v$ be the length-$N$ suffix of $w$.
    By induction hypothesis, $v$ occurs once in $x_{\llbracket-N+1,N\rrbracket}$. If $v$ is not left-special, any occurrence of $v$ in $x$ corresponds to an occurrence of $w$. So, $w$ occurs exactly once in $x_{\llbracket-N,N\rrbracket}$, and as $v \ne \ell_N(x) = x_{\llbracket 2,N+1 \rrbracket}$, $w$ occurs exactly once in $x_{\llbracket-N,N+1\rrbracket}$. 
    
    Consequently, the only length-$(N+1)$ factor that might have no occurrence in $x_{\llbracket-N,N+1\rrbracket}$ is $0 \ell_N(x)$. Using the factor complexity of $x$, this would imply that $1 \ell_N(x)$ has exactly two occurrences. With the same argument on right-special factors, the only factor that might have no occurrence is $r_N(x) 1$, and in this case $r_N(x) 0$ has exactly two occurrences. This would imply that $0 \ell_N(x) = r_N(x) 1$ and $1 \ell_N(x) = r_N(x) 0$. However, this is impossible as it implies that $r_N(x)$ begins with both $0$ and $1$. We therefore conclude that any length-$(N+1)$ factor appears exactly once.
\end{proof}

When only considering Sturmian words, the previous proposition gives the following characterisation of characteristic words.

\begin{definition}[Characteristic Sturmian word]\label{D:characteristic words}
A Sturmian word $x$ is \emph{upper} (resp., \emph{lower}) \emph{characteristic} if
\[
    x_{\llbracket -n, n+1\rrbracket} = r_n(x)01\ell_n(x) \text{\quad (resp., } x_{\llbracket -n, n+1\rrbracket} = r_n(x)10\ell_n(x) \text{) \quad for all } n.
\]
It is \emph{characteristic} if it is lower or upper characteristic.
\end{definition}

\begin{corollary}\label{C:Sturmian words of span 1}
An aperiodic word has the string attractor $\llbracket 0,1 \rrbracket$ if and only if it is a characteristic Sturmian word.
Consequently, an aperiodic word $x$ has a string attractor of span $1$ if and only if it is, up to finite shift, a characteristic Sturmian word. Moreover, this string attractor is perfect.
\end{corollary}

In~\cite{Barbieri_Labbe_Starosta}, Barbieri, Labbé and Starosta  implicitly obtain one direction of this corollary. Namely, they prove that characteristic Sturmian words form a so-called non-trivial indistinguishable asymptotic pair with difference set $\llbracket0,1\rrbracket$, and as such, they have the string attractor $\llbracket 0,1 \rrbracket$. Their proof relies on the interpretation of Sturmian words as mechanical words of irrational slope. The technique used here is inspired by a combinatorial result of~\cite{Zamboni}.

To understand string attractors of non-characteristic Sturmian words, we make use of the classical $S$-adic description of Sturmian words, see~\cite{arnoux_fogg_2002}.

\begin{proposition}\label{P:Sturmian S-adic representation}
Let $x$ be a Sturmian word. There exist a non-eventually constant sequence $(a_i)_{i \in \N} \in \{0,1\}^\N$, a sequence of Sturmian words $(x^{(i)})_{i \in \N}$, and a sequence $(c_i)_{i \in \N} \in \N^\N$ such that, for all $i$, $x = S^{c_i}L_{a_0} \cdots L_{a_i} (x^{(i)})$ where
\[
    L_0\colon
    \begin{cases}
        0 \mapsto 0 \\
        1 \mapsto 01
    \end{cases}
    \text{and}
    \quad
    L_1\colon
    \begin{cases}
        0 \mapsto 10 \\
        1 \mapsto 1
    \end{cases}\!\!\!.
\]
\end{proposition}

Observe that $L_0$ and $L_1$ are return morphisms for $0$ and $1$ respectively. It is also well-known that $L_0$ and $L_1$ preserve characteristic words.

\begin{lemma}\label{L:images of characteristic words}
    If $x$ is an upper (resp., lower) characteristic word, then $S L_0(x)$ and $S L_1(x)$ are two upper (resp., lower) characteristic words.
\end{lemma}

\begin{proof}
    We consider the substitution $L_0$; the proof for $L_1$ is obtained by swapping $0$ with $1$.
    Recall first that $S L_0(x)$ is Sturmian, see~\cite[Lemma 6.3.5]{arnoux_fogg_2002} for example.
    Applying $L_0$ to $x$, we have, for all $n$, $L_0(x_{\llbracket -n,n+1\rrbracket}) = L_0(r_n(x)) L_0(x_{\llbracket 0,1\rrbracket}) L_0(\ell_n(x))$.
    By Lemma~\ref{L:return morphisms}, for all $m$, $r_m(L_0(x))$ is the length-$m$ suffix of $L_0(r_n(x))0$ and $\ell_m(L_0(x))$ is the length-$m$ prefix of $L_0(\ell_n(x))$ for $n$ large enough. Moreover, $L_0(01) = 001$ and $L_0(10) = 010$. Therefore, $L_0(x)_{\llbracket -m+1,0\rrbracket} = r_m(L_0(x))$, $L_0(x)_{\llbracket 1,2\rrbracket} = x_{\llbracket 0,1\rrbracket}$ and $L_0(x)_{\llbracket 3,m+2\rrbracket} = \ell_m(L_0(x))$ for all $m$.
    This shows that $SL_0(x)$ is upper (respectively, lower) characteristic whenever $x$ is.
\end{proof}

As return morphisms, $L_0$ and $L_1$ are also well-behaved with respect to string attractors. The following link between $\spn(\Gamma)$ and $\spn(\Par{L_0}{y}(\Gamma))$ follows from the definition of the substitutions. We write $|w|_1$ the number of occurrences of $1$ in the word $w$.

\begin{lemma}\label{L:length of inverse image L_0}
    Let $y \in \{0,1\}^\Z$. If $\Gamma = \llbracket n, n+k \rrbracket$ and $\Par{L_0}{y}(\Gamma) = \llbracket m, m + \ell \rrbracket$, then $\ell = k - \left|L_0(y)_{\llbracket n+1,n+k\rrbracket}\right|_1$. 
    The same result holds for $L_1$ by replacing $1$ by $0$.
\end{lemma}

\begin{proof}
    By definition of $\Par{L_0}{y}$, we have $\Gamma = \bigcup_{i\in \llbracket m, m+\ell \rrbracket} (\IM{L_0}{y}(i) \cap \Gamma)$ where the union is disjoint and none of the intersections are empty.
    Note that $\#\IM{L_0}{y}(i) = 1$ when $y_i=0$ and $\#\IM{L_0}{y}(i) = 2$ when $y_i = 1$.
    Therefore, $\#(\IM{L_0}{y}(i) \cap \Gamma) = 1$ unless $L_0(y_i) = 01$ and the second position is in $\llbracket n+1,n+k\rrbracket$, in which case the intersection contains two elements. By definition of $L_0$, this implies that $k+1 = \# \Gamma = \ell+1 + \left|L_0(y)_{\llbracket n+1,n+k\rrbracket}\right|_1$.
\end{proof}

We then refine Lemma~\ref{L:inverse image of sa return morphism} as follows:

\begin{lemma}\label{L:inverse image of sa L_0}
    Let $y \in \{0,1\}^\Z$ and $x = L_0(y)$. Assume that $\Gamma = \llbracket n, n+k \rrbracket$ is a string attractor of $x$ and let $\Par{L_0}{y}(\Gamma) = \llbracket m, m + \ell \rrbracket$.
    \begin{enumerate}
    \item
        The set $\llbracket m-1, m+\ell\rrbracket$ is a string attractor of $y$. In particular, $\spn(y) \leq k + 1 - \left|x_{\llbracket n+1,n+k\rrbracket}\right|_1$.
    \item
        If $x_{n-1}x_n \ne 00$, then $\llbracket m, m+\ell\rrbracket$ is a string attractor of $y$. In particular, $\spn(y) \leq k - \left|x_{\llbracket n+1,n+k\rrbracket}\right|_1$.
    \item
        If $x_{n+k}x_{n+k+1} = 01$, then $\llbracket m-1, m+\ell-1\rrbracket$ is a string attractor of $y$. In particular, $\spn(y) \leq k - \left|x_{\llbracket n+1,n+k\rrbracket}\right|_1$.
    \item
        If $x_{n-1}x_n \ne 00$ and $x_{n+k}x_{n+k+1} = 01$, then $\llbracket m, m+\ell-1\rrbracket$ is a string attractor of $y$. In particular, $\spn(y) \leq k - 1 - \left|x_{\llbracket n+1,n+k\rrbracket}\right|_1$.
    \end{enumerate}
    
The same result holds for $L_1$ by switching $0$ and $1$ everywhere.
\end{lemma}
\begin{proof}
    The first claim is Lemma~\ref{L:inverse image of sa return morphism}. The bound on the span follows from Lemma~\ref{L:length of inverse image L_0} since $\left|x_{\llbracket n+1,n+k\rrbracket}\right|_1 = k - \ell$.
    The proof of the other claims relies repeatedly on the fact that a $1$ is always preceded and followed by a $0$ in $x$.
    
    For the second claim, we have either $x_{n-1}x_n = 01$ or $10$. In the first case, $\Par{L_0}{y}(\llbracket n-1, n+k \rrbracket) = \llbracket m, m+\ell\rrbracket$ so we conclude using Lemma~\ref{L:inverse image of sa return morphism}.
    If $x_{n-1}x_n = 10$ and if $w$ has an occurrence in $y$ ending at position $m-1$, then $w$ ends with $1$. The occurrences of $w$ in $y$ correspond to the occurrences of $L_0(w)0$ in $x$, and since $L_0(w)$ ends with $1$, they correspond to the occurrences of $L_0(w)$ in $x$. As $\Gamma$ is a string attractor of $x$, $L_0(w)$ has an occurrence crossing $\Gamma$ in $x$. Therefore, $w$ also has an occurrence crossing $\Par{L_0}{y}(\Gamma)$ in $y$. This ends the proof of the second claim.
    
    For the third claim, if $x_{n+k}x_{n+k+1} = 01$ and if $w$ has an occurrence in $y$ starting at position $m+\ell$, then $w$ starts with $1$. Write $w = 1u$. The occurrences of $1u$ in $y$ correspond to the occurrences of $01L_0(u)0$ in $x$, which themselves correspond to the occurrences of $1L_0(u)0$. As $\Gamma$ is a string attractor of $x$, $1L_0(u)0$ has an occurrence crossing $\llbracket n, n+k\rrbracket$, so $L_0(w) = 01L_0(u)$ crosses $\llbracket n-1, n+k-1 \rrbracket$ in $x$ and $w$ crosses $\Par{L_0}{y}(\llbracket n-1, n+k-1 \rrbracket)$ in $y$. This shows that $\Par{L_0}{y}(\llbracket n-1, n+k-1\rrbracket)$ is a string attractor of $y$. As $\Par{L_0}{y}(n+k-1) = m + \ell -1$, this ends the proof of the third claim.
    
    For the last claim, if $x_{n-1}x_n = 01$, then $\Par{L_0}{y}(\llbracket n-1,n+k-1\rrbracket) = \llbracket m,m+\ell-1\rrbracket$, so the previous paragraph directly shows that $\llbracket m, m+\ell-1\rrbracket$ is a string attractor of $y$. If $x_{n-1}x_n = 10$, then using the two previous claims, both $\llbracket m, m+\ell\rrbracket$ and $\llbracket m-1, m+\ell-1\rrbracket$ are string attractors. To show that $\llbracket m, m+\ell-1\rrbracket$ is a string attractor, it suffices to consider the words having an occurrence ending in $m-1$ and an occurrence starting in $m+\ell$. As above, such a word $w$ starts and ends with $1$, and the occurrences of $w$ in $y$ correspond to the occurrences of $1L_0(u)$ in $x$, where $w = 1u$. As $1L_0(u)$ crosses $\llbracket n, n+k\rrbracket$, $L_0(w)$ crosses $\llbracket n,n+k-1\rrbracket$, and $w$ is captured by $\Par{L_0}{y}(\llbracket n,n+k-1\rrbracket) = \llbracket m, m+\ell-1\rrbracket$.
\end{proof}

We now fully describe the span of Sturmian words.

\begin{theorem}\label{T:Sturmian no span other than 1}
    If $x$ is a Sturmian word, then
    \[
        \spn(x) = 
        \begin{cases}
            1, & \text{if, up to finite shift, $x$ is characteristic;}\\
            \infty, & \text{otherwise.}
        \end{cases}
    \]
\end{theorem}

\begin{proof}
    The fact that $\spn(x) = 1$ if and only if $x$ is, up to finite shift, a characteristic Sturmian word is Corollary~\ref{C:Sturmian words of span 1}.
    It remains to show that, if $x$ is Sturmian and it admits a finite string attractor, then $\spn(x) = 1$. Let $\Gamma = \llbracket n, n+k \rrbracket$ be a string attractor of $x$. If $k = 1$, we directly have $\spn(x) = 1$ so assume that $k \geq 2$.
    
    By Proposition~\ref{P:Sturmian S-adic representation}, there exist a non-eventually constant sequence $(a_i)_{i \in \N} \in \{0,1\}^\N$, a sequence of Sturmian words $(x^{(i)})_{i \in \N}$ and a sequence $(c_i)_{i \in \N} \in \N^\N$ such that, for all $i$, $x = S^{c_i}(L_{a_0} \cdots L_{a_i}(x^{(i)}))$.
    In particular, $\spn(x) = \spn(L_{a_0} \cdots L_{a_i}(x^{(i)}))$ for every $i$. As we are looking at the span and not the actual string attractor, we assume for simplicity's sake that $x = L_{a_0} \cdots L_{a_i}(x^{(i)})$ for every $i$.
    
    We first prove that there exists $N$ such that $\spn(x^{(N)}) < k$.
    Assume that $x = L_0(x^{(0)})$ (the other case is symmetric) and write $\llbracket m, m+\ell \rrbracket = \Par{L_0}{x^{(0)}}(\Gamma)$.
    Lemma~\ref{L:inverse image of sa L_0} implies that $x^{(0)}$ is of span at most $k-1$ (in which case we take $N = 0$) unless we are in one of the following two cases:
    \begin{enumerate}
        \item $\left|x_{\llbracket n+1, n+k \rrbracket}\right|_1 = 1$, $x_{n-1}x_n = 00$, and $x_{n+k}x_{n+k+1} \ne 01$; in which case $x^{(0)}$ admits the string attractor $\llbracket m-1, m+\ell \rrbracket$ with $\ell = k-1$. Moreover, we have $x_{\llbracket n,n+k\rrbracket} = 0^i10^j$ with $i \geq 1$, and if $j \geq 1$, then $x_{n+k+1} = 0$. This implies that $x^{(0)}_{\llbracket m-1,m+k-1\rrbracket} = 0^{i}10^j = x_{\llbracket n, n+k\rrbracket}$.
        \item $\left|x_{\llbracket n+1, n+k \rrbracket}\right|_1 = 0$ and ($x_{n-1}x_n = 00$ or $x_{n+k}x_{n+k+1} \ne 01$); in which case $x^{(0)}$ admits the string attractor $\llbracket m, m+\ell \rrbracket$ with $\ell = k$. Indeed, as $\llbracket n,n+k\rrbracket$ captures an occurrence of $1$, we must have $x_n = 1$, so $x_{n-1}x_n \ne 00$. Moreover, by hypothesis, $x_{n+k}x_{n+k+1} \ne 01$ so $x_{n+k+1} = 0$. This implies that $x^{(0)}_{\llbracket m,m+k\rrbracket} = 10^k = x_{\llbracket n, n+k\rrbracket}$.
    \end{enumerate}
    
    In both cases, $|x_{\llbracket n, n+k \rrbracket}|_1 = 1$ and $x^{(0)}$ has a string attractor $\llbracket m', m'+k \rrbracket$ such that $x^{(0)}_{\llbracket m', m'+k \rrbracket} = x_{\llbracket n, n+k \rrbracket}$.
    
    Since substitutions $L_0$ and $L_1$ alternate infinitely often, let $N$ be such that $x = L_0^{N}(x^{(N-1)})$ and $x^{(N-1)} = L_1(x^{(N)})$. 
    Iterating the reasoning above, $x^{(N-1)}$ admits the string attractor $\llbracket p, p+k\rrbracket$ for some $p$ such that $|x^{(N-1)}_{\llbracket p,p+k\rrbracket}|_1 = 1$. 
    Use the same reasoning for the substitution $L_1$ by exchanging the roles of $0$ and $1$: as $|x^{(N-1)}_{\llbracket p,p+k\rrbracket}|_0 = k \geq 2$, we do not fall into one of the two particular cases and obtain $\spn(x^{(N)}) < k$.

    By infinite descent, this shows the existence of $M$ such that $\spn(x^{(M)}) = 1$. By Corollary~\ref{C:Sturmian words of span 1}, $x^{(M)}$ is then, up to finite shift, a characteristic Sturmian word. Iterating Lemma~\ref{L:images of characteristic words}, we deduce that $x$ is also, up to finite shift, a characteristic Sturmian word, so $\spn(x) = 1$ again by Corollary~\ref{C:Sturmian words of span 1}.
\end{proof}

\subsection{The quasi-Sturmian case}

In this section, we generalise the link between Sturmian words and span $1$ to quasi-Sturmian words and any finite span. Let us make explicit the link between quasi-Sturmian and Sturmian words. A substitution $\varphi$ over a binary alphabet $\{a,b\}$ is \emph{acyclic} if $\varphi(ab) \ne \varphi(ba)$, or equivalently if $\varphi(a)$ and $\varphi(b)$ are not integer powers of the same word.

\begin{theorem}[\cite{Cassaigne_grouped_factors}]\label{T:quasi-Sturmian characterisation}
    A word $x$ is quasi-Sturmian if and only if there exist an acyclic substitution $\varphi$, an integer $m$, and a Sturmian word $y$ such that $x = S^m \varphi(y)$.
\end{theorem}

A close inspection of the proof of this result provides extra information on $\varphi$. As we need this slightly stronger version, we detail the proof found in~\cite{these_Alessandri} or~\cite{Cassaigne_grouped_factors} which relies on Rauzy graphs.

\begin{definition}[Rauzy graph]
Let $x$ be a bi-infinite word. The \emph{Rauzy graph of rank $n$}, $n \geq 0$, of $x$ is the graph $\cG_n(x)$ whose vertices are the length-$n$ factors of $x$, and there is an edge $u\xrightarrow{a} v$ labeled by $a\in\cA$ if and only if there exists a letter $b\in \cA$ such that $ub = av$ is a length-$(n+1)$ factor of $x$.
\end{definition}

In particular, for all $n \geq 0$, $x$ labels an infinite walk on the graph $\cG_n(x)$.
Observe also that, in $\cG_n(x)$, a path from $w$ to $v$ labelled by $u$ is such that $w$ is a prefix of $uv$. If moreover, $w = v$ and the path does not go through the vertex $w$, i.e., the path forms a loop on the vertex $w$, then $uw$ contains exactly two occurrences of $w$: at the beginning and at the end. In other words, $u$ is a \emph{return word} for $w$. Recall that return morphisms are defined so that images of letters correspond to return words.

Combining Rauzy graphs and return morphisms, we obtain the following result. Observe that, as a quasi-Sturmian word also admits exactly one left-special and one right-special word of length $n$ for large values of $n$, we extend the notation $\ell_n(x)$ and $r_n(x)$ to a quasi-Sturmian word $x$.

\begin{proposition}\label{P:quasi-Sturmian morphism}
    Let $x$ be a quasi-Sturmian word. There exist a bispecial factor $w$ of $x$, a substitution $\varphi$, a Sturmian word $y \in \{0,1\}^\Z$, and an integer $m$ such that $x = S^m \varphi(y)$ and the following conditions are satisfied:
    \begin{enumerate}
    \item $\varphi$ is a return morphism for $w$;
    \item $w0$ is a prefix of $\varphi(0)w$, $w1$ is a prefix of $\varphi(1)w$, and $\varphi(0)$ and $\varphi(1)$ end with different letters;
    \item $p_x(n) = n + |\varphi(0)| + |\varphi(1)| - |w| - 1$ for all $n \geq |w|$.
    \end{enumerate}
\end{proposition}

\begin{proof}
    Let $k$ and $n_0$ be such that $p_x(n) = n + k$ for all $n \geq n_0$.
    In other words, $r_n(x)$ and $\ell_n(x)$ are defined for all $n \geq n_0$.

    There is an index $n_1 \geq n_0$ such that $\ell_{n_1}(x)$ is bispecial. Indeed, assume by contradiction that, for all $n \geq n_0$, $\ell_n(x)$ is not bispecial, so that the unique extension of $\ell_n(x)$ to the right is $\ell_{n+1}(x)$.
    By uniform recurrence of $x$, $\ell_{n_0}(x)$ has infinitely many occurrences in $x$, so there is an index $n$ such that $\ell_{n_0}$ is a prefix and a suffix of $\ell_n(x)$.
    Since the extension is unique at each step, we find that $x_i = x_{i+n-n_0}$ for $i$ large enough, a contradiction with the aperiodicity of $x$.
    
    Let us define $w = \ell_{n_1}(x)$ and consider the Rauzy graph $\mathcal{G}_{n_1}(x)$. By construction, every vertex except $w$ has one ingoing and one outgoing edge, and $w$ has two of each. It follows that $\mathcal{G}_{n_1}(x)$ is made of two loops starting and ending in $w$; see Figure~\ref{fig:graph}.
    Let us define $\varphi$ such that $\varphi(0)$ and $\varphi(1)$ are the labels of these two loops.

    For $a \in \{0,1\}$, since $\varphi(a)$ labels a loop on $w$, $w$ occurs in $\varphi(a)w$ only as a prefix and a suffix.
    As the loops enter $w$ through different edges, $\varphi(0)$ and $\varphi(1)$ end with different letters, and as they leave through different edges, the prefix $w$ is followed by different letters in $\varphi(0)w$ and $\varphi(1)w$. In particular, $\varphi$ is injective on $\{0,1\}$, and it is therefore a return morphism. Without loss of generality, we can assume that $wa$ is a prefix of $\varphi(a)w$.
    
    \begin{figure}[h]
    \centering
    \input{proof_Rauzy_graph}
    \caption{ \label{fig:graph}The graph $\mathcal{G}_{n_1}(x)$.}
    \end{figure}
    
    Since $x$ labels an infinite walk on $\mathcal{G}_{n_1}(x)$, there exist a word $y$ and an integer $m$ such that $x = S^m\varphi(y)$.
    Let us show that $y$ is Sturmian. Since $x$ is aperiodic, $y$ is as well. We show that $y$ has at most one left-special factor of each length, which implies that $p_y(n) = n+1$ for all $n$.
    
    Assume by contradiction that $y$ has two length-$n$ left-special factors $t_0$ and $t_1$. We assume that $n$ is minimal, so that there is only one length-$(n-1)$ left-special factor $t'$ which is then a prefix of both $t_0$ and $t_1$. Without loss of generality, assume that $t_a = t'a$ for $a \in \{0,1\}$. Since $t_a$ is left-special in $y$, $\varphi(t')\varphi(a)w$ is left-special in $x$ by Lemma~\ref{L:return morphisms}. As $wa$ is a prefix of $\varphi(a)w$, this shows that both $\varphi(t')w0$ and $\varphi(t')w1$ are left-special in $x$. This contradicts the fact that $x$ only has one left-special factor of each length $n \geq |w|$.
    
    Finally, let us show that $p_x(n) = n + |\varphi(0)| + |\varphi(1)| - |w| - 1$ for all $n \geq n_0$. In the construction of the substitution $\varphi$, $p_x(n_1)$ is the number of vertices of the graph $\cG_{n_1}(x)$, and $|\varphi(0)|$ and $|\varphi(1)|$ are the number of vertices in each loop, the loops having exactly one common vertex. It follows that $p_x(n_1) = |\varphi(0)| + |\varphi(1)| - 1$. Since $|w| = n_1 \geq n_0$ and $p_x(n) = n + k$ for all $n \geq n_0$, this ends the proof.
\end{proof}

The role played by characteristic Sturmian words in the previous section leads us to the following generalisation.

\begin{definition}[Characteristic quasi-Sturmian word]\label{D:characteristic qS}
A quasi-Sturmian word $x$ is \emph{characteristic} if there exists a finite word $u$ such that:
\[
    x_{\llbracket -n, n+|u|-1\rrbracket} = r_n(x)u\ell_n(x) \textrm{\quad for all $n$ large enough that $r_n(x)$ and $\ell_n(x)$ are defined}.
\]
\end{definition}

Note that, while not explicitly named, characteristic quasi-Sturmian words also appear in \cite[Theorem C]{Barbieri_Labbe_Starosta} in relation with indistinguishable asymptotic pairs.

\begin{remark}\label{R:shifted of characteristic}
Two characteristic quasi-Sturmian words cannot be equal up to finite shift. Indeed, if both $x$ and $S^m(x)$ are characteristic quasi-Sturmian words, then, as they have the same right-special factors, $x_{-i} = x_{m-i}$ for all $i \in \N$. By aperiodicity of quasi-Sturmian words, we must have $m = 0$.
\end{remark}

\begin{proposition}\label{P:characteristic quasi-Sturmian}
A bi-infinite word $y$ is characteristic quasi-Sturmian if, and only if, $y = S^{|w|}\varphi(x)$ where $x$ is a characteristic Sturmian word, $w$ is a bispecial factor of $y$ and $\varphi$ is a return morphism for $w$.
\end{proposition}

\begin{proof} 
Assume that $y$ is characteristic quasi-Sturmian.
By Proposition~\ref{P:quasi-Sturmian morphism}, there exists a return morphism $\varphi$ for a bispecial $w$ such that $y$ is, up to shift, the image of a Sturmian word by $\varphi$. More precisely, as $w$ is bispecial, it is a suffix of $r_n(y)$, so $S^{-|w|}(y) = \varphi(x)$ where $x$ is a Sturmian word. Let us show that $x$ is characteristic.

By Lemma~\ref{L:return morphisms}, for all large enough $m$, $\varphi(r_m(x))w$ is a suffix of $y_{\llbracket -n, -1\rrbracket}$
for any large enough $n$. Therefore, $x_{\llbracket -m,-1\rrbracket} = r_m(x)$ for all $m$. Write $z$ the characteristic Sturmian word such that $\cL(z) = \cL(x)$ and $x_0 = z_0$. If $x \ne z$, there exists $i \geq 1$ minimal such that $x_i \ne z_i$. Thus, $r_m(x)x_{\llbracket 0, i-1\rrbracket}$ is a right-special word, and $r_m(x)x_{\llbracket 0, i-1\rrbracket} = r_{m+i}(x)$ for all $n$. This contradicts the aperiodicity of $x$. Therefore, $x=z$ so $x$ is a characteristic Sturmian word. 

Assume now that $y = S^{|w|}(\varphi(x))$ where $x$ is characteristic Sturmian, $w$ is bispecial in $y$ and $\varphi$ is a return morphism for $w$. As $w$ is bispecial, $w$ is the longest common prefix to $\varphi(0)w$ and $\varphi(1)w$, and $\varphi(0)$ and $\varphi(1)$ end with different letters.
Observe that, as $\varphi(01) \ne \varphi(10)$, $|w| < |\varphi(01)|$ so write $\varphi(x_{\llbracket 0,1\rrbracket}) = wu$.
For all large enough $n$, $r_n(y)$ is a suffix of $\varphi(r_n(x))w$ and $\ell_n(y)$ is a prefix of $\varphi(\ell_n(x))$ by Lemma~\ref{L:return morphisms}. This shows that $y_{\llbracket -n, n+|u|-1\rrbracket} = r_n(y)u\ell_n(y)$ for all large enough $n$, so $y$ is characteristic.
\end{proof}

We then have the counterpart of Theorem~\ref{T:Sturmian no span other than 1}.

\begin{theorem}\label{T:QS possible spans}
    Let $x$ be a quasi-Sturmian word such that $p_x(n) = n+k$ for all large enough $n$. We have
    \[
        \spn(x) = 
        \begin{cases}
            k, & \text{if, up to finite shift, $x$ is characteristic;}\\
            \infty, & \text{otherwise.}
        \end{cases}
    \]
\end{theorem}

\begin{proof}
    By Proposition~\ref{P:quasi-Sturmian morphism}, we have a return morphism $\varphi$ for a bispecial word $w$, a Sturmian word $y$, and an integer $m$ such that $x = S^m \varphi(y)$ and $k = |\varphi(0)| + |\varphi(1)| - |w| - 1$.
    Assume that $\spn(x)$ is finite. By Lemma~\ref{L:inverse image of sa return morphism}, $y$ is also of finite span. By Theorem~\ref{T:Sturmian no span other than 1}, $\spn(y) = 1$ and $y$ is a characteristic Sturmian word (up to finite shift). Hence, the positions of the string attractor of $y$ correspond to $01$ or $10$. By Lemma~\ref{L:image of sa return morphism}, we have $\spn(x) \leq |\varphi(0)| + |\varphi(1)| - |w| - 1 = k$. Using the link between span and factor complexity (Proposition~\ref{P:finite SA bound comp}), this implies that $\spn(x) = k$ and $x$ is, up to finite shift, the image of a characteristic Sturmian word by a substitution. Thus $x$ is, up to finite shift, a characteristic quasi-Sturmian word by Proposition~\ref{P:characteristic quasi-Sturmian}.
    Conversely, if $x$ is the image of a characteristic Sturmian word, then it is of finite span by Theorem~\ref{T:Sturmian no span other than 1} and Proposition~\ref{P:under_morphism}.
\end{proof}

\begin{corollary}\label{qs-string-attractor}
An aperiodic word $x$ has the string attractor $\llbracket 0, k \rrbracket$ that is eventually perfect if and only if it is a characteristic quasi-Sturmian word of complexity $p_x(n) = n + k$ for any large enough $n$.
\end{corollary}
\begin{proof}
If $\llbracket 0, k \rrbracket$ is an eventually perfect string attractor of $x$, then for any large enough $n$, we have $p_x(n) = n+k$. In particular, $x$ is quasi-Sturmian. Write $u = x_{\llbracket 0,k\rrbracket}$. Using the same proof as in Proposition~\ref{P:unicity of words of span 1}, we can show that $x_{\llbracket -n, n+k\rrbracket} = r_n(x)u\ell_n(x)$ for all large enough $n$. Thus, $x$ is characteristic.

Conversely, if $x$ is a quasi-Sturmian characteristic word of complexity $p_x(n) = n + k$ for any large enough $n$, it admits a string attractor of span $k$ by Theorem~\ref{T:QS possible spans}. Therefore, there exists $m \in \Z$ such that $S^m(x)$ admits the string attractor $\llbracket 0, k \rrbracket$, and this string attractor is eventually perfect based on the complexity of $S^m(x)$. Using the first part of the proof, $S^m(x)$ is also characteristic, which implies that $m = 0$ by Remark~\ref{R:shifted of characteristic}. Therefore $x$ admits the string attractor $\llbracket 0,k\rrbracket$.
\end{proof}

Observe that the notion of (eventually) perfect string attractor is related to that of (ultimately) grouped factors introduced in~\cite{Cassaigne_grouped_factors}. Indeed, if $x$ has an (eventually) perfect string attractor then, for all (large enough) $n$, it has a length-$(p_x(n)+n-1)$ factor containing all length-$n$ factors exactly once, meaning that it has (ultimately) grouped factors. The converse is false however as all quasi-Sturmian words have ultimately grouped factors.

\subsection{Final characterisation}
Putting together the results of this section, we obtain the following characterisation of bi-infinite words admitting a finite string attractor.

\begin{theorem}\label{T:finite string attractors}
A bi-infinite word $x$ has a finite string attractor if and only if it falls in one of the following (disjoint) cases:
\begin{itemize}
    \item it is purely periodic of minimal period $p$, in which case $p_x(n) = p > \spn(x)$ for any large enough $n$;
    \item it is eventually periodic but not purely periodic, in which case $p_x(n) = n + \spn(x)$ for any large enough $n$;
    \item it is, up to finite shift, characteristic quasi-Sturmian, in which case $p_x(n) = n + \spn(x)$ for any large enough $n$.
\end{itemize}
\end{theorem}
\begin{proof}
By Propositions~\ref{P:finite SA bound comp} and~\ref{P:Coven bi-infinite words}, if $x$ admits a finite string attractor, then $x$ falls into one of the following mutually exclusive cases: $x$ is purely periodic, $x$ is eventually periodic but not purely periodic, or $x$ is quasi-Sturmian.

All periodic words admit a finite string attractor, of span smaller than the minimal period by Remark~\ref{R:periodic}.

All eventually periodic (non-purely periodic) words admit a finite string attractor, and moreover satisfy $p_x(n) = n + \spn(x)$ for any large enough $n$ by Proposition~\ref{P:sa of periodic words}.

Finally, among the quasi-Sturmian words, only those that are characteristic up to finite shift  admit a finite string attractor, and moreover satisfy $p_x(n) = n + \spn(x)$ for any large enough $n$ by Theorem~\ref{T:QS possible spans}.
\end{proof}

\subsection{Shift spaces with finite string attractors}

Let us now translate our results in terms of shift spaces.

\begin{definition}
    A shift space is \emph{Sturmian} if it is the orbit closure of a Sturmian word, and \emph{quasi-Sturmian} if it is the orbit closure of a quasi-Sturmian word. It is in particular minimal.
\end{definition}

It is well known that every Sturmian shift space contains exactly two characteristic words. The same is true for every quasi-Sturmian shift space: they are the image of a Sturmian shift space by the return morphism given by Proposition~\ref{P:quasi-Sturmian morphism}. Thus, any characteristic quasi-Sturmian word is in the orbit of the image of one of the two characteristic Sturmian words by Theorem~\ref{T:Sturmian no span other than 1}, Proposition \ref{P:quasi-Sturmian morphism}, and Lemma~\ref{L:inverse image of sa return morphism}. By Remark~\ref{R:shifted of characteristic} and Proposition~\ref{P:characteristic quasi-Sturmian}, each such orbit contains exactly one characteristic quasi-Sturmian word.

\begin{proposition}\label{P:finite SA for QS}
    An infinite minimal shift space $\XX$ contains a word $x \in \XX$ of finite span if, and only if, $\XX$ is quasi-Sturmian. In this case, it contains countably infinitely many words with finite span, and uncountably many with infinite span.
\end{proposition}

\begin{proof}
If $\XX$ is infinite, minimal, and contains $x \in \XX$ of finite span, then $x$ must be aperiodic (otherwise $\XX$ would be finite). By Theorem~\ref{T:finite string attractors}, $x$ is quasi-Sturmian, and since $\XX$ is minimal, $\XX$ is quasi-Sturmian as the orbit closure of $x$.

Conversely, if $\XX$ is quasi-Sturmian, then it contains two characteristic quasi-Sturmian words which are of finite span by Theorem~\ref{T:QS possible spans}, as well as all of their (infinitely countably many) shifts. All other words in $\XX$ have infinite span by Theorem~\ref{T:finite string attractors}. 
\end{proof}

It follows that only periodic shift spaces have finite string attractors, falling back to a behavior similar to one-sided infinite words.

\begin{proposition}
    A shift space $\XX$ has a finite string attractor if and only if $\XX$ is finite, that is, $\XX$ is a finite union of periodic shift spaces.
\end{proposition}

\begin{proof}
    It is known that a shift space contains only purely periodic words if and only if it is finite, see~\cite[Theorem 3.8]{Jeandel}.

    Assume that $\mathbb X$ is generated by a purely $p$-periodic word $x$.
    By periodicity, $\llbracket 0, p-1 \rrbracket$ is a string attractor of $\mathbb X$. For a finite union, take the largest such string attractor.
    
    Conversely, let $\mathbb X$ be a shift space with a finite string attractor and let $x \in \XX$. By Theorem~\ref{T:finite string attractors}, $x$ is quasi-Sturmian or eventually periodic. If $x$ is quasi-Sturmian, then $\overline{\mathcal{O}(x)}$ is a minimal quasi-Sturmian shift space so, by Proposition~\ref{P:finite SA for QS}, there is a point in $\overline{\mathcal{O}(x)} \subseteq \mathbb X$ with infinite span, which contradicts the existence of a finite string attractor of $\XX$. Therefore, $x$ is eventually periodic. Furthermore, $x$ is also uniformly recurrent. Indeed, if $\llbracket \gamma, \gamma+k \rrbracket$ is a string attractor of $\XX$, and in particular of every $S^j(x)$, then for all $n \geq 1$, every length-$n$ factor of $x$ occurs in every length-$(2n + k)$ factor of $x$. By Remark~\ref{R:uniformly recurrent and aperiodic imply periodic}, $x$ is then purely periodic, and as this is true for any $x \in \XX$, this shows that $\XX$ only contains purely periodic words. 
\end{proof}

\subsection{Pattern attractors}\label{sec:pattern-attractors}

Let us discuss some additional properties of the finite string attractors obtained so far. We call \emph{pattern} of $x\in\mathcal A^{\Z}$ a function $p\colon S \to \mathcal A$, where $S\subset \mathbb N$ is a finite set containing $0$, such that, for some index $i$, $x_{i+s} = p(s)$ for all $s \in S$.
Note that every non-empty factor is a pattern (where $S$ is an interval). If $S = \{s_0, \dots, s_n\}$ with $0 = s_0 < s_1 < \cdots < s_n$, we write $p$ as $p(s_0)\$^{s_1-s_0-1}p(s_1)\$^{s_2-s_1-1}\cdots p(s_n)$ where $\$$ is a new symbol representing the fact that the position is not in $S$.

As mentioned previously, in \cite{Barbieri_Labbe_Starosta}, the authors characterise indistinguishable asymptotic pairs, that is, pairs $(x,y)$ of bi-infinite words that are equal except on a finite \emph{difference set} and so that every pattern crosses this difference set the same number of times in $x$ and in $y$. For aperiodic words, these pairs correspond to characteristic quasi-Sturmian words (up to finite shift). Their proof uses the fact that each member of an indistinguishable pair admits the difference set as a string attractor. In~\cite[Section 2.2]{Barbieri_Labbe_Starosta}, it seems implied (without proof) that these string attractors are in fact \emph{pattern attractors}. We provide a proof below.

\begin{definition}
    Let $x$ be a word and $\Gamma$ be a set of positions within $x$. A pattern $p\colon S \to \mathcal A$ is \emph{captured by $\Gamma$} if there exists an index $i$ such that $x_{i+s} = p(s)$ for all $s\in S$ and $\Gamma \cap (i + S) \ne \emptyset$.
    If every pattern of $x$ is captured by $\Gamma$, we say that $\Gamma$ is a \emph{pattern attractor} of $x$.
\end{definition}

\begin{example}\label{ex:pattern-attract}
Let $x = \cdots001010010\underline{10}0100101001\cdots$ be the Fibonacci word which is lower characteristic Sturmian. Positions $\llbracket 0,1\rrbracket$ (underlined) form a string attractor. They also capture the pattern $0\$\$1$ once (by position $0$), and the pattern $0\$\$\$ 0$ twice (by position $1$).
\end{example}

\begin{proposition}\label{prop:pattern}
A bi-infinite aperiodic word has a finite pattern attractor if and only if it is, up to finite shift, a characteristic quasi-Sturmian word. In this case, the finite string attractor given in Corollary~\ref{C:Sturmian words of span 1} or Corollary~\ref{qs-string-attractor} is also a pattern attractor.
\end{proposition}

\begin{proof}
As any pattern attractor is a string attractor, it is sufficient to show that any characteristic quasi-Sturmian word admits the required finite pattern attractor.

Let $y$ be a characteristic quasi-Sturmian word.  By Proposition~\ref{P:characteristic quasi-Sturmian}, let $w$ be a bispecial factor of $y$ and $\varphi$ be a return morphism for $w$ such that $y = S^{|w|}\varphi(x)$ for some characteristic Sturmian word $x$. Assume moreover that $\varphi(0)w$ starts with $w0$ and $\varphi(1)w$ starts with $w1$ as in Proposition~\ref{P:quasi-Sturmian morphism}. Let us assume that $x$ is upper characteristic, the proof for lower characteristic is symmetric.

By Corollary~\ref{qs-string-attractor} and Proposition~\ref{P:quasi-Sturmian morphism}, $\llbracket 0, |\varphi(01)| - |w| - 1\rrbracket$ is a string attractor of $y$. Let us show that it also captures patterns. Let $u$ be a pattern occurring in $y$. 
A \emph{completion} of $u$ is a finite word obtained by replacing each symbol $\$$ in $u$ by a letter. Among all completions of $u$ that are factors of $y$, let $u^\uparrow$ denote the highest completion in the lexicographic order. By hypothesis, $u^\uparrow$ is captured by $\llbracket 0, |\varphi(01)| - |w| - 1\rrbracket$. This occurrence of $u^\uparrow$ is also an occurrence of $u$, and if any position in $\llbracket 0, |\varphi(01)| - |w| - 1\rrbracket$ corresponds to a non-$\$$ symbol in this occurrence of $u$, then $u$ is also captured. Otherwise, $u^\uparrow = vy_{\llbracket 0, |\varphi(01)| - |w| - 1\rrbracket}v'$ where $v$ is right-special, $v'$ is left-special and $u_{\llbracket |v|, |u|-|v'|-1 \rrbracket} = \$^{|\varphi(01)| - |w|}$. 

Since $x$ is upper characteristic, we have $y_{\llbracket-|v|, |u|-|v|-1\rrbracket} = u^\uparrow = vt^{\downarrow}v'$ where $\varphi(01) = wt^{\downarrow}$. Denoting $x'$ the corresponding lower characteristic word, we have $(S^{|w|}\varphi(x'))_{\llbracket -|v|, |u|-|v|-1\rrbracket} = vt^{\uparrow}v'$ where $\varphi(10) = wt^{\uparrow}$. It follows that $vt^\uparrow v'$ is another completion of $u$ that is a factor of $y$. As we assumed that $t^{\downarrow}$ starts with $0$ and $t^{\uparrow}$ starts with $1$, we found another completion of $u$ that is higher than $u^\uparrow$, a contradiction. We conclude that every pattern $u$ occurring in $y$ is captured by $\llbracket 0, |\varphi(01)| - |w| - 1\rrbracket$.
\end{proof}

The pattern attractor in Proposition~\ref{prop:pattern} is not perfect since some patterns can be captured multiple times, as Example~\ref{ex:pattern-attract} illustrates.

\begin{corollary}
A bi-infinite word has a finite pattern attractor if and only if it has a finite string attractor.
\end{corollary}
\begin{proof}
The equivalence for aperiodic words follows the previous proposition and Theorem~\ref{T:finite string attractors}. For eventually periodic words, consider $x = S^k(\cdots uuu. w vvv\cdots)$ assuming w.l.o.g that $|u| = |v|$ (by replacing $u$ by $u^{|v|}$ and $v$ by $v^{|u|}$ if necessary). Then the set $\llbracket -|u| - k, |wv| - k-1\rrbracket$ is a pattern attractor: indeed, any occurrence of a pattern can be shifted by $\pm|u|$ into another occurence of the same pattern until it intersects this set.
\end{proof}

While a finite pattern attractor is trivially a finite string attractor, we did not find a direct proof that a word with a finite string attractor has a finite (possibly larger) pattern attractor; we relied on the fact that these two notions are equivalent to being a characteristic quasi-Sturmian word or eventually periodic. It may be a hint that these notions are fundamentally distinct and coincide in dimension 1 by happenstance. The situation is similar for aperiodic words that are members of an indistinguishable asymptotic pair; regarding eventually periodic words, Example~\ref{example:unbalanced} provides a word that has a finite pattern attractor but is not a member of such a pair. 

\section{Infinite string attractors}\label{S:infinite string attractors}

A natural question is whether anything can be said about words and shift spaces that do not admit any finite string attractor, such as most words in a Sturmian shift, the Prouhet--Thue--Morse word, or any infinite shift space. For example, we may search for string attractors with the lowest density $\limsup\limits_n \frac{\#(\Gamma \cap \llbracket-n,n\rrbracket)}{2n+1}$ or log-density $\limsup\limits_n \frac{\#(\Gamma\cap \llbracket-n,n\rrbracket)}{\log(2n+1)}$. The following result tells us that the lowest such value is always $0$ in the case of recurrent words; in other words, there always are arbitrarily sparse string attractors.

\begin{proposition}\label{P:Density of infi SA}
    Let $x$ be a bi-infinite recurrent word, and let $\eta$ be an arbitrary non-decreasing function such that $\lim\limits_n \eta(n) = + \infty$.
    Then $x$ has a string attractor $\Gamma$ such that $\#(\Gamma \cap \llbracket -n, n \rrbracket) \leq \eta(n)$ for all $n$.
\end{proposition}

\begin{proof}
We fix an enumeration of the non-empty factors of $x$, that is, a bijection $\textrm{enum}\colon\N\setminus\{0\} \to \mathcal L(x) \setminus \{\varepsilon\}$.
For all $i$, we define $\gamma_i$ as the position in $\mathbb Z$ closest to $0$ such that $\eta(|\gamma_i|)\geq i$ and an occurrence of $\textrm{enum}(i)$ crosses the position $\gamma_i$ in $x$. Such a position always exists since $x$ is recurrent. Notice that the same position may be chosen multiple times.
The set $\Gamma = (\gamma_i)_{i \geq 1}$ is clearly a string attractor. Moreover, if $\gamma_k \in \llbracket -n,n \rrbracket$, then by hypothesis on $\eta$ and by definition of $\gamma_k$, we have $k \leq \eta(|\gamma_k|) \leq \eta(n)$. This shows that $\Card(\Gamma \cap \llbracket -n, n \rrbracket) \leq \eta(n)$.
\end{proof}

\begin{remark}
The same proof shows that any recurrent one-sided infinite word has arbitrarily sparse string attractors.
\end{remark}
A similar result applies to minimal shift spaces:

\begin{proposition}\label{P:sparse string attractor for shift spaces}
    Let $\XX$ be a minimal shift space, and let $\eta$ be an arbitrary non-decreasing function such that $\lim\limits_n \eta(n) = + \infty$.
    Then $\XX$ has a string attractor $\Gamma$ such that $\#(\Gamma \cap \llbracket -n, n \rrbracket) \leq \eta(n)$ for all $n$.
\end{proposition}

\begin{proof}
As in the word case, let us fix an enumeration of the non-empty factors of $\XX$, that is, a bijection $\textrm{enum}\colon \N \setminus \{0\} \to \cL(\XX)\setminus \{\varepsilon\}$. By uniform recurrence, let $c_i$ be a constant such that any length-$c_i$ interval contains an occurrence of $\textrm{enum}(i)$ in every $x\in\XX$.
Now fix $\Gamma = \bigcup_i \llbracket\gamma_i, \gamma_i+c_i -1\rrbracket$ with $\gamma_i$ chosen large enough that $\#(\Gamma \cap \llbracket -n, n \rrbracket) \leq \eta(n)$ for all $n$. By definition of $c_i$, $\textrm{enum}(i)$ has an occurrence included in $\llbracket \gamma_i, \gamma_i+c_i-1\rrbracket$ in every $x\in\XX$, so $\Gamma$ is a string attractor of $\XX$.
\end{proof}

Therefore, we cannot find a minimal or sparsest string attractor. We turn our attention to the existence of string attractors with a specific structure. In what follows, we study words and shift spaces having all (non-trivial) arithmetic progressions as string attractors. To do so, we introduce the following notation.

\begin{definition}[Occurrences $\bmod\ k$]
Given $k \geq 1$, $w\in \cA^*$ and $x\in \cA^\Z$, the set of starting positions $\bmod\ k$ of occurrences of $w$ in $x$ is denoted by
\[
    \Occ{k}{x}{w} = \{i \bmod k \mid x_{\llbracket i, i+|w|-1\rrbracket} = w\}.
\]
\end{definition}

Clearly, the set $i+k\Z$ is a string attractor of $x$ if and only if $\Occ{k}{x}{w} \cap \{(i-j) \bmod k : j \in \llbracket 0, |w|-1 \rrbracket\} \neq \emptyset$ for every non-empty factor $w$ of $x$. This suggests the following stronger definition, found in~\cite{Cassaigne_modulo_recurrent}.

\begin{definition}[Modulo-recurrence] A word $x$ is modulo-recurrent if, for every $w \in \cL(x)$ and every $k \geq 1$, $\Occ{k}{x}{w} = \Z/k\Z$.
\end{definition}

\begin{proposition}\label{P:modulo-recurrent and sa}
If $x$ is modulo-recurrent, then every arithmetic progression $i+k\Z$ is a string attractor of $x$, for $i\in\mathbb Z$ and $k \geq 1$. 
\end{proposition}

From the point of view of shift spaces, modulo-recurrence corresponds to a property called total minimality that was known much earlier (see~\cite{Paul} for example), just as uniform recurrence corresponds to minimality.

\begin{definition}
A shift space $\mathbb X$ is totally minimal if it is minimal under the action of $S^k$ for all $k\geq 1$.
\end{definition}

\begin{proposition}
The shift space $\mathbb X$ is totally minimal if and only if it is the orbit closure of a uniformly recurrent modulo-recurrent word. In particular, all words of $\mathbb X$ are modulo-recurrent.
\end{proposition}

\begin{proof}
Take $x \in \mathbb X$. Because $\mathbb X$ is minimal for $S$, $x$ is uniformly recurrent. By contradiction, assume that $0 < \Card{\Occ{k}{x}{w}} < k$ for some $k>0$ and factor $w$; then $\Occ{k}{S(x)}{w} = \Occ{k}{x}{w} - 1 \not \subseteq \Occ{k}{x}{w}$, in $\Z/k\Z$. The set of words $y$ such that $\Occ{k}{y}{w} \subseteq \Occ{k}{x}{w}$ is $S^k$-invariant, nonempty, closed, and not equal to $\XX$ (it does not contain $S(x)$), so $\XX$ is not minimal for $S^k$, a contradiction. This shows that any $x \in \XX$ is modulo-recurrent.

For the other direction, take $x$ a uniformly recurrent modulo-recurrent word. Its orbit closure (under $S$) $\XX$ is minimal. In particular, the words of $\XX$ all share the same language. To show that $\XX$ is minimal under the action of $S^k$, it is enough to show that $\XX$ is the orbit closure under $S^k$ of any $y \in \XX$. Let $w \in \cL(x)$. Since $x$ is modulo-recurrent, there is a factor $u$ of $x$ in which $w$ occurs starting at every position modulo $k$. Since $\XX$ is minimal, $u$ occurs in $y$, so $\Occ{k}{y}{w} = \mathbb Z/k\mathbb Z$. It follows that $S^{nk}(y)_{[0, |w|-1]} = w$ for some $n\in\mathbb Z$. This holds for every $w \in \cL(x)$, therefore every $z\in \XX$ is in the orbit closure of $y$ under $S^k$ (by taking $w = z_{\llbracket-n,n\rrbracket}$ for all $n$).
\end{proof}

We obtain the following translation of Proposition~\ref{P:modulo-recurrent and sa}.

\begin{proposition}\label{P:kZ universal aperiodic}
    If $\mathbb X$ is a totally minimal shift space, then the sets $i+k\Z$ are string attractors of $\XX$ for every $i\in\Z$, $k \geq 1$.
\end{proposition}

Sturmian words are shown to be modulo-recurrent in~\cite{Kabore_Tapsoba}, and Sturmian shift spaces to be totally minimal in~\cite{Paul}. More generally, we have the following result.

\begin{proposition}[{\cite[Proposition 12]{Berthe.et.al}}]
Minimal dendric shift spaces are totally minimal.
Recurrent dendric words are modulo-recurrent.
\end{proposition}

\begin{corollary}
Recurrent dendric words, and Sturmian words in particular, admit all arithmetic progressions as string attractors.

Minimal dendric shift spaces, and Sturmian shift spaces in particular, admit all arithmetic progressions as string attractors.
\end{corollary}

On the other hand, quasi-Sturmian words may or may not be modulo-recurrent. If they are not, they may or may not admit every arithmetic progression as a string attractor, as we show below. This implies that modulo-recurrence (and total minimality) are not equivalent to having all arithmetic progressions as string attractors.

\begin{proposition}\label{prop:total_occurrences_examples}
Let $x \in \{0,1\}^\Z$ be any Sturmian word.
\begin{enumerate}
\item
Let $\psi$ be the substitution defined by $\psi(0) = 01$ and $\psi(1) = 00$. The set $2\Z$ is not a string attractor of $\psi(x)$.
\item
Let $\varphi$ be the substitution defined by $\varphi(0) = 01$ and $\varphi(1) = 10$. The word $\varphi(x)$ is not modulo-recurrent but it has every arithmetic progression as a string attractor.
\end{enumerate}
\end{proposition}
\begin{proof}
The first claim is direct: every even position in $\psi(x)$ contains a $0$, so the factor $1$ is not captured by $2\Z$.

We now consider $y = \varphi(x)$. Without loss of generality, we assume that $00$ is a factor of $x$ but $11$ is not. Therefore, $01$ is a factor of $y$ but it only occurs starting at even positions, so $\Occ{2}{y}{01} = \{0\}$. This shows that $y$ is not modulo-recurrent.

Now take $i\in\Z$, $k \geq 1$, and a non-empty factor $w$ of $y$. First assume that $|w| = 1$ and consider the set of positions $(i+k\Z) \cap 2\Z$. If it is nonempty, then it is an infinite arithmetic progression in $y$, which corresponds to an arithmetic progression in $x$ seeing the same letters. By modulo-recurrence of $x$, the letter $w$ is then captured by $i+k\Z$ in $y$. If $(i+k\Z) \cap 2\Z$ is empty, do the same argument with $(i+k\Z) \cap (2\Z+1)$ and $\overline{x}$ where $\overline{x}_n = 1-x_n$.

Now assume that $|w| > 1$. By adding at most one letter to the left and to the right, $w$ can be extended to $w' = \varphi(u)$ where $u$ is a factor of $x$. The word $x$ is modulo-recurrent, so for all $k$, $\Occ{k}{x}{u} = \Z/k\Z$. This implies that $\Occ{k}{y}{w'}$ contains at least all even elements of $\Z/k\Z$. We have either $\Occ{k}{y}{w} \supset \Occ{k}{y}{w'}$ or $\Occ{k}{y}{w} \supset \Occ{k}{y}{w'}+1$, so $\Occ{k}{y}{w}$ contains all even or all odd elements. It follows that $i$ or $i-1$ is in $\Occ{k}{y}{w}$, and since $|w|>1$, $i+k\mathbb Z$ captures $w$.
\end{proof}

\begin{remark}
The proof can be extended to show that, if $x = \varphi(y)$ with $y$ modulo-recurrent and $\varphi$ a $\ell$-uniform substitution, then every $i+k\Z$ captures every factor of length at least $\ell$ in $x$. To determine whether $i+k\Z$ is a string attractor, it is enough to check smaller factors.
\end{remark}

We end with a comment on the intuition behind modulo-recurrence and total minimality: the word or shift space does not have any hidden periodic structure. Hidden structure can be uncovered through cellular automata.

\begin{definition}\label{D:cellular automata}
    A map $\pi\colon \cA^\Z \to \cB^\Z$ is a cellular automaton if there exist an integer $M \geq 1$ and a map $\psi\colon \cA^M \to \cB$ such that $\pi(x)_n = \psi(x_{\llbracket n, n+M-1\rrbracket})$ for all $n \in \Z$.
\end{definition}

See~\cite{Lind_Marcus} for example for other equivalent definitions.
In the context of shift spaces, $\pi$ is also called a \emph{factor map} from $\XX$ to $\pi(\XX)$. While the following result is known for shift spaces (see for example \cite{Paul} for the difficult direction), we provide a combinatorial proof in the context of a single word.

\begin{proposition}\label{P:total uniform recurrence and periodic factors}
A uniformly recurrent bi-infinite word $x$ is modulo-recurrent if and only if it is completely aperiodic, i.e., for any cellular automaton $\pi$, $\pi(x)$ is aperiodic or the constant word.

A minimal shift space $\XX$ is totally minimal if and only if it is completely aperiodic, i.e., for any cellular automaton $\pi$, $\pi(\XX)$ is strongly aperiodic (does not contain any purely periodic word) or contains a single constant word.
\end{proposition}

\begin{proof}
We only prove this result in the case of a uniformly recurrent word $x$. We proceed by contraposition.
Assume first that $x$ is not completely aperiodic, and there exists a cellular automaton $\pi$ such that $\pi(x)$ is neither aperiodic nor constant. If $x$ is uniformly recurrent, so is $\pi(x)$. Therefore, by Remark~\ref{R:uniformly recurrent and aperiodic imply periodic}, $\pi(x)$ is purely periodic. Let $k$ be its period and $0 < i < k$ be such that $\pi(x)_i \ne \pi(x)_0$. If $\psi\colon \cA^M \to \cB$ is the local rule corresponding to $\pi$ by Definition~\ref{D:cellular automata}, and if $w = x_{\llbracket 0, M-1\rrbracket}$, then $\psi(w) = \pi(x)_0$ and, for all $j \equiv i \pmod k$, we have $\psi(x_{\llbracket j, j+M-1\rrbracket}) = \pi(x)_j = \pi(x)_i \ne \psi(w)$. This shows that $i \not \in \Occ{k}{x}{w}$ so $x$ is not modulo-recurrent.

Conversely, assume that $i\notin\Occ{k}{x}{w}$ for some $w \in \cL(x)$ and $0 \leq i < k$. Let us find a cellular automaton $\pi$ such that $\pi(x)$ is purely periodic but not constant.
For each $j \in \Occ{k}{x}{w}$, let us fix $n_j \in \Z$ such that $n_j \equiv j \pmod k$ and $x_{\llbracket n_j, n_j + |w| - 1\rrbracket} = w$. We define $N = |w| + \max\{|n_j| \mid j \in \Occ{k}{x}{w}\}$ and $u = x_{\llbracket-N, N\rrbracket}$. Let $M$ be the uniform recurrence bound for $u$, i.e., every length-$M$ factor of $x$ contains an occurrence of $u$.

Observe that, for every position $m \in \Z$, we have
\[
    \{j \bmod k \mid j \in \llbracket m, m + M - |w|\rrbracket \text{ and } x_{\llbracket j, j+|w|-1\rrbracket} = w\} \subseteq \Occ{k}{x}{w}.
\]
Moreover, since $x_{\llbracket m,m+M-1\rrbracket}$ contains an occurrence of $u$, the two sets have the same cardinality so they are equal. Consider the following local rule:
\[
    \psi\colon\left(\begin{array}{ccc}
        \cA^M & \rightarrow & \{0,1\} \\
        v & \mapsto & \begin{cases}
            1, & \text{ if there is an occurrence of $w$ in $v$ whose position is a multiple of $k$;}\\
            0, & \text{ otherwise.}
        \end{cases}
    \end{array}\right)
\]
By definition, if $\psi(x_{\llbracket m, m+M-1\rrbracket}) = 1$, then $(m \bmod k) \in \Occ{k}{x}{w}$. Conversely, if $(m \bmod k) \in \Occ{k}{x}{w}$, then by the observation above, there exists $j \in \llbracket m, m + M - |w|\rrbracket$ such that $m \equiv j \pmod k$ and $x_{\llbracket j, j+|w|-1\rrbracket} = w$. In other words, $\psi(x_{\llbracket m, m+M-1\rrbracket}) = 1$.

If $\pi$ is the cellular automaton associated with $\psi$ as in Definition~\ref{D:cellular automata}, then $\pi(x)_m = 1$ if and only if $(m \bmod k) \in \Occ{k}{x}{w}$. This implies that the word $\pi(x)$ is purely periodic, but as there exists $i \not \in \Occ{k}{x}{w}$, it is not constant.
\end{proof}

\section{Further work}

We completely characterised words on $\Z$ that admit a finite string attractor, and we found some properties of words that admit every arithmetic progression as string attractor. Are there any other infinite sets or structures such that words that admit this set as string attractor form a natural or interesting class? In a different direction, could the notion of pattern attractor prove useful in characterising other families of finite words, similarly to the existing literature on string attractors?

The definition of string (or pattern) attractor can be extended for higher-dimensional words, as done in~\cite{2D}. One easily sees that, as in dimension 1, having a finite attractor heavily constrains the factor complexity. In the process of generalising indistinguishable asymptotic pairs to higher dimensions in~\cite{barbieri2025indistinguishable}, Barbieri and Labbé provide examples of multidimensional Sturmian words that admit finite string attractors and their results hint that it could be a characterisation in a restricted case: see Question 1 in~\cite{barbieri2025indistinguishable}. We hope these notions prove useful in improving our understanding of low-complexity words which are not well-understood in higher dimension, as is highlighted by the infamous Nivat's Conjecture. 

\acknowledgements
We thank Francesco Dolce and Giuseppe Romana for initial discussions on the topic and anonymous reviewers for many constructive comments.

During part of this research, France Gheeraert was a Research Fellow of the Fonds de la Recherche Scientifique - FNRS.
This research was supported by the ANR-22-CE40-0011 Inside Zero Entropy Systems.

\bibliographystyle{abbrvnat}
\bibliography{bibliography}

\end{document}

%% file: image_return_1.tex
\begin{tikzpicture}
    \draw (-2,0.5) node{$\varphi(x)w =$};
    \draw[dashed] (-1,0)--(0,0);
    \draw[dashed] (-1,1)--(0,1);
    \draw (0,0)--(1.5,0);
    \draw (0,1)--(1.5,1);
    \draw (4.5,1)--(1.5,1)--(1.5,0)--(4.5,0);
    \draw[fill=gray!30] (4.5,0)--(7.5,0)--(7.5,1)--(4.5,1);
    \draw (4.5,0.5) node{$\varphi(x_{\llbracket n, n+k\rrbracket})$};
    \draw (7.5,0)--(8.5,0);
    \draw (7.5,1)--(8.5,1);
    \draw[dashed] (8.5,0)--(10,0);
    \draw[dashed] (8.5,1)--(10,1);
    
    \draw[decoration={brace, mirror, raise = 0.05cm, amplitude = 5pt}, decorate] (0.2,1)-- node[below = 0.15cm, midway]{$u$} (4,1);
    
    \draw[decoration={brace, raise = 0.05cm, amplitude = 5pt}, decorate] (0,1)-- node[above = 0.15cm, midway]{$\varphi(t)$} (2.5,1);
    \draw[decoration={brace, raise = 0.05cm, amplitude = 5pt}, decorate] (2.5,1)-- node[above = 0.15cm, midway]{$w$} (5.5,1);
    
    \draw[decoration={brace, mirror, raise = 0.05cm, amplitude = 5pt}, decorate] (0,0)-- node[below = 0.15cm, midway]{$\varphi(t')$} (1.5,0);
    \draw[decoration={brace, mirror, raise = 0.05cm, amplitude = 5pt}, decorate] (1.5,0)-- node[below = 0.15cm, midway]{$w$} (4.5,0);
\end{tikzpicture}

%% file: proof_Rauzy_graph.tex
\begin{tikzpicture}[node distance=1cm, main/.style={circle, draw, minimum size=.5cm}, bezier bounding box]
    \node[main] (w) at (0,0) {$w$};
    \node[minimum size=0cm, inner sep=0cm] (w0) [above right of=w] {};
    \node[minimum size=0cm, inner sep=0cm] (w1) [above left of=w] {};
    \draw[thick] (w) edge[<-] node[near start, above] {$a$} (w0);
    \draw[thick] (w0) edge [<-, out=45,in=315, looseness=5, dotted] node[left]{$\varphi(0)$} (w);
    \draw[thick] (w) edge[<-] node[near start, above] {$\overline{a}$} (w1);
    \draw[thick] (w1) edge [<-, out=135,in=225, looseness=5, dotted] node[midway, right]{$\varphi(1)$} (w);
\end{tikzpicture}